\documentclass[reqno]{amsart}

\usepackage{amssymb}
\usepackage{amsfonts}
\usepackage{amsmath}
\usepackage{color}
\usepackage{amsaddr}

\newtheorem{theorem}{Theorem}[section]

\newtheorem{corollary}[theorem]{Corollary}

\newtheorem{definition}[theorem]{Definition}

\newtheorem{lemma}[theorem]{Lemma}

\newtheorem{proposition}[theorem]{Proposition}

\theoremstyle{remark}

\newtheorem{remark}[theorem]{Remark}

\numberwithin{equation}{section}

\begin{document}

\title[Boundary Conditions with Memory]{Weak Exponential Attractors for Coleman--Gurtin Equations with Dynamic Boundary Conditions Possessing Different Memory Kernels}

\author[J. L. Shomberg]{Joseph L. Shomberg}



\address{Department of Mathematics and Computer Science, Providence
College, Providence, RI 02918, USA, \\
\tt{{jshomber@providence.edu} }}

\date{\today}

\begin{abstract}
The well-posedness of a generalized Coleman--Gurtin equation equipped with dynamic boundary conditions with memory was recently established by C.G. Gal and the author. 
Additionally, it was established by the author that the problem admits a finite dimensional global attractor and a robust family of exponential attractors in the case where singularly perturbed memory kernels defined on the interior of the domain and on the boundary of the domain coincide.
In the present article we report advances concerning the asymptotic behavior of this heat transfer model when the memory kernels do {\em not} coincide. 
In this setting we obtain a weak exponential attractor whose basin of attraction is the entire phase space, that is, a finite dimensional exponentially attracting compact set in the weak topology of the phase space. 
This result completes an analysis of the finite dimensional attractors for the generalized Coleman--Gurtin equation equipped with dynamic boundary conditions with memory.
\end{abstract}

\maketitle

\tableofcontents

\section{Introduction to the model problem}

This article presents the asymptotic behavior of a heat transfer model with memory effects. 
Following in the spirit of Gurtin and Mizel's seminal contribution \cite{Coleman&Mizel63} is the article \cite{Gal-Shomberg15-2} where the author provides a derivation of a thermodynamic process taking place in a bounded container/domain $\Omega\subset\mathbb{R}^d$ under the five basic constitutive functions: specific internal energy, specific entropy, heat flux, absolute temperature and heat supply.
Novel to this derivation however is the assumption that heat may be supplied from the exterior of the domain; namely, from the boundary of the container. 
This type of contribution is not present in \cite{Coleman&Mizel63}, nor the more recent \cite{CDGP-2010}.
We also consider the departure from Fourier's heat law where the heat fluxes involve memory to account for inertial effects.
This, importantly, prevent instantaneous signal propagation that is associated with the standard heat equation.

The principle equations under consideration in this article are from \cite{Gal-Shomberg15-2} and are the following:
\begin{align}
\partial_t u - \omega\Delta u - (1-\omega) \int_0^\infty k_\Omega(s) \Delta u(x,t-s) {\rm d}s + f(u) + \alpha(1-\omega) \int_0^\infty k_\Omega(s) u(x,t-s) {\rm d}s =0  \label{eq1m}
\end{align}
in $\Omega \times (0,\infty)$ subject to the boundary condition
\begin{align}
\partial_t u - \omega\Delta_\Gamma u + \omega\partial_{\bf n} u & + (1-\omega) \int_0^\infty k_\Gamma(s)\partial_{\bf n} u(x,t-s){\rm d}s  \notag \\
& + (1-\omega) \int_0^\infty k_\Gamma(s)(-\Delta_\Gamma+\beta)u (x,t-s){\rm d}s + g(u) = 0  \label{eq2m} 
\end{align}
on $\Gamma \times (0,\infty)$ for every $\alpha,\beta\ge0$, $\omega \in [0,1)$, and where $k:[0,\infty)\rightarrow \mathbb{R}$ is a continuous nonnegative function, smooth on $(0,\infty)$, vanishing at infinity and satisfying the relation
\begin{equation*}
\int_0^\infty k_S(s) {\rm d}s = 1, \quad S\in\{\Omega,\Gamma\},
\end{equation*}
$\partial_{\bf n}$ represents the normal derivative and $-\Delta_\Gamma$ is the Laplace--Beltrami operator.
The cases $\omega =0$ and $\omega >0$ in (\ref{eq1m}) are usually referred
as the Gurtin--Pipkin and the Coleman--Gurtin models, respectively. 
Let $\omega \in [0,1)$ be fixed. 
Notice that if we (formally) choose $k_S=\delta_0$ (the Dirac mass at zero), equations (\ref{eq1m})-(\ref{eq2m}) turn into the following system,
\begin{equation}
\partial_tu - \Delta u + f(u) + \alpha(1-\omega)u = 0, \quad \text{in}\ \Omega \times (0,\infty),  \label{eq3d}
\end{equation}
\begin{equation}
\partial_tu - \Delta_\Gamma u + \partial_{\bf n} u + g(u) + \beta(1-\omega) u = 0, \quad \text{on}\ \Gamma \times (0,\infty). \label{eq4d}
\end{equation}

Memory used as a ``hyperbolic-like relaxation'' term appears in \cite{CGG11} where the authors consider a viscous Cahn--Hilliard equation with dynamic dynamic boundary conditions.
Dynamic boundary conditions can be used to account for frictional damping or involve sources/sinks on the boundary.
In addition, the Cahn--Hilliard equation serves as a motivating example for using dynamic boundary conditions.
The typical static Neumann (standard conserving) boundary condition invokes the unnatural property that the interface separating two phases be orthogonal to the boundary.
This quality does not necessarily appear with dynamic boundary conditions.
Several other phase-field type equations with memory and dynamic boundary conditions appear in \cite{Conti-Zelati10} and \cite{Gal&Grasselli12}.

A tremendous amount of recent activity in dissipative dynamical systems comes from applications that admit attractors.
By definition, dissipative dynamical systems possess a bounded absorbing set to which any nonempty bounded subset is attracted and absorbed in some finite time.
Absorbing sets lack further important descriptions such as compactness and finite dimensionality. 
Exponential attractors (also called inertial sets) are formalized in \cite{EFNT95}.
Although not unique, exponential attractors are finite dimensional compact subsets that exponentially attract nonempty bounded subsets of the phase space. 
In several applications (cf. e.g. \cite{CPS06}) the basin of attraction of the exponential attractor can be tied to the existence of a unique global attractor (also called universal attractors); namely, the asymptotic compactness property of the solution operators. 
Contrary to exponential attractors, global attractors do not enjoy the finite dimensionality description.
Concerning problems involving memory, some recent applications include \cite{GGPS05,GMPZ10,Grasselli&Pata05,KloedenRealSun2011,Plinio&Pata09}.

It is the article \cite{Pata&Zelik06-2} where we first see the construction of the so-called weak exponential attractor. 
The compactness enjoyed by the attractor is inherited from a more regular absorbing set which itself is compact in the weak phase space.
The ``higher-order'' estimates required to obtain the compact absorbing set are not available for the type of wave equation examined in \cite{Pata&Zelik06-2}.
Complicating matters in our presentation is the functional formulation of the memory terms.
Indeed, compact embeddings between memory spaces is a delicate issue that is discussed below.
So in some contexts it is better to build off the absorbing set that it obtained in the standard (weak {\em energy}) phase space.
Hence, a weak exponential attractor inherits its compactness by the compact injection into the weak topology of the standard phase space.

The article \cite{CDGP-2010} presents the Coleman--Gurtin equation with Dirichlet boundary conditions and nonlinear terms allowing critical growth. 
The authors show the existence of global attractors, as well as exponential attractors, with optimal regularity.
Some improvements to these results, though for nonlinearities under the critical limit, appear in \cite{Gal-Shomberg15-2} and \cite{Shomberg-reacg16}.
Indeed, \cite{Gal-Shomberg15-2} contains a treatment on global weak solutions.
Strong solutions are developed under the assumption that the memory kernel on the interior matches the memory kernel on the boundary. 
Additionally, in the more general setting when the two kernels do not coincide, we develop quasi-strong solutions. 
These quasi-strong solutions have just enough regularity to allow a compact embedding of the memory space into a lower-ordered space whereby allowing the construction of a weak exponential attractor. 
To clarify our motivation, consider Banach spaces $X, Y, W, Z, M$, such that $W\subset X$, $Z\subset Y$ and $M\subset Z$ continuously, and $M\subset Y$ compactly, and where $X'$ and $Y'$ denote the dual spaces of, respectively, $X$ and $Y$.
In this generic setting, our story unfolds as follows: for appropriate spaces, the global weak solutions found in \cite{Gal-Shomberg15-2} may be set in $X\times Y$ and the global quasi-strong solutions in $W\times Z$.
Moreover, one component of the quasi-strong solution is uniformly {\em bounded} in $W$, and we are able to further show that the global quasi-strong solutions also belong to $W\times M$; i.e., to a more regular space in the second component.
Since $M$ is compactly contained in $Y$, we deduce, using the adjoint of the standard embeddings $W\subset X$ and $Z\subset Y$, that $W\times M$ is compact in $X'\times Y'.$
Finally, it is for these global quasi-strong solutions that we obtain an exponential attractor in the topology of $X'\times Y'$.

With respect to the above development, we carefully treat the following issues:

\begin{description}
\item[1] Well-posedness of the system comprising of equations (\ref{eq1m})-(\ref{eq2m}) and (\ref{eq3d})-(\ref{eq4d}).
We include results to the global existence of weak solutions and so-called quasi-strong solutions.
There is no restriction to the size of the initial datum and the global weak solutions generate a Lipschitz continuous semigroup of solution operators that is uniformly continuous in time on compact intervals.

\item[2] Dissipation exhibited by the model problem in the sense that the solution operators admit a bounded absorbing set in the weak energy phase space.
We do not assume that memory kernels coincide.
Since $k_\Omega\not=k_\Gamma$, we do eventually need to implement an assumption on the size of $k_\Gamma(0)$ in order to perform standard analysis arguments.
At this point it is crucial to obtain an absorbing set in a phase space that includes a memory component that can be compactly imbedded into a lower-order memory component.
This step is where we require the use of quasi-strong solutions.
Precisely, we are showing the existence of an absorbing set admitted by the quasi-strong solutions that embeds compactly in the weak topology of the phase space for the weak solutions.

\item[3] The existence of a weak exponential attractor (a finite dimensional compact attractor in the weak topology of the phase space) for the model problems.
Due to a transitivity of exponential attraction result, we know the basin of attraction of the exponential attractor is the entire weak energy phase space equipped with the weak topology. 
For PDEs with memory in dynamic boundary conditions, it seems that this is the first general result where a problem exhibits a (weak) exponential attractor.
\end{description}

Following recent convention (cf. e.g. \cite{CDGP-2010,CPS05,CPS06,Grasselli&Pata02-2}) we introduce the so-called integrated past history of $u$, i.e., the auxiliary variable
\begin{equation*}
\eta^t(x,s) = \int_0^s u(x,t-\sigma) {\rm d}\sigma,
\end{equation*}
for $s,t>0.$ 
Next, by setting 
\begin{equation*}  
\mu_S(s) = -(1-\omega)k'_S(s), \quad S\in\{\Omega,\Gamma\},
\end{equation*}
formal integration by parts into (\ref{eq1m}) and (\ref{eq2m}) yields
\begin{align}
(1-\omega) \int_0^\infty k_S(s) u(x,t-s) {\rm d}s & = \int_0^\infty \mu_S(s) \eta^t(x,s) {\rm d}s, \quad S\in\{\Omega,\Gamma\}, \notag \\
(1-\omega) \int_0^\infty k_\Omega(s) \Delta u(x,t-s) {\rm d}s & =\int_0^\infty\mu_\Omega(s) \Delta \eta^t(x,s) {\rm d}s, \notag \\ 
(1-\omega) \int_0^\infty k_\Gamma(s) \partial_{\bf n} u(x,t-s) {\rm d}s & = \int_0^\infty \mu_\Gamma(s) \partial_{\bf n} \eta^t(x,s) {\rm d}s, \notag 
\end{align}
and
\begin{equation*}
(1-\omega) \int_0^\infty k_\Gamma(s)(-\Delta_\Gamma + \beta) u(x,t-s) {\rm d}s = \int_0^\infty \mu_\Gamma(s) (-\Delta_\Gamma + \beta) \eta^t(x,s) {\rm d}s. 
\end{equation*}

We now state

Problem {\bf P}:
Let $\alpha,\beta\ge0,$ and $\omega\in(0,1)$. 
Find real-valued functions $(u,\eta)=(u(x,t),\eta^t(x,s))$ defined in $\Omega\times(0,\infty)$ such that 
\begin{equation}
\partial_t u - \omega\Delta u - \int_0^\infty \mu_\Omega(s) \Delta\eta^t(x,s) {\rm d}s + \alpha\int_0^\infty \mu_\Omega(s) \eta^t(x,s) {\rm d}s + f(u) = 0  \label{problemp-1}
\end{equation}
in $\Omega\times(0,\infty),$ subject to the boundary conditions
\begin{equation}
\partial_t u - \omega\Delta_\Gamma u + \omega\partial_{\bf n} u + \int_0^\infty \mu_\Gamma(s)\partial_{\bf n} \eta^t(x,s) {\rm d}s + \int_0^\infty \mu_\Gamma(s) (-\Delta_\Gamma + \beta) \eta^t(x,s) {\rm d}s + g(u) = 0  \label{problemp-2}
\end{equation}
on $\Gamma\times(0,\infty),$ and 
\begin{equation}
\partial_t \eta^t(x,s) + \partial_s \eta^t(x,s) = u(x,t) \quad \text{in}\ {\overline{\Omega}}\times(0,\infty),  \label{problemp-3}
\end{equation}
with
\begin{equation}
\eta^t(x,0) = 0 \quad \text{in}\ {\overline{\Omega}}\times(0,\infty),  \label{problemp-4}
\end{equation}
and the initial conditions 
\begin{equation}
u(x,0) = u_0(x) \quad \text{in}\quad \Omega, \quad u(x,0) = v_0(x) \quad \text{on}\ \Gamma,  \label{problemp-5}
\end{equation}
\begin{equation}
\eta^0(x,s) = \eta_0(x,s):=\int_0^su_0(x,-y)dy \quad\text{in}\ \Omega, \quad \text{for}\ s>0,  \label{problemp-6}
\end{equation}
and
\begin{equation}
\eta^0(x,s) = \xi_0(x,s):=\int_0^sv_0(x,-y)dy \quad\text{on}\ \Gamma, \quad \text{for}\ s>0.  \label{problemp-7}
\end{equation}

\begin{remark}  \label{on-traces}
It need not be the case that the boundary traces of $u_0$ and $\eta_0$ be equal to $v_0$ and $\xi_0$, respectively. 
Thus, we are solving a much more general problem in which equation (\ref{problemp-1}) is interpreted as an evolution equation in the interior $\Omega$ properly coupled with the equation (\ref{problemp-2}) on the boundary $\Gamma$.
Finally, according to Definition \ref{d:weak-solution}, we regard both $\eta_0$ and $\xi_0$ as being independent of the initial data $u_0$ and $v_0.$
Indeed, below we will consider a more general problem with respect to the original one. 
\end{remark}

We now give the framework used to prove Hadamard well-posedness for Problem {\textbf{P}}. 
Consider the space $\mathbb{X}^2:=L^2(\overline{\Omega}, d\mu) ,$ where
\begin{equation*}
{\rm d}\mu = {\rm d}x_{\mid \Omega }\oplus {\rm d}\sigma ,
\end{equation*}
where ${\rm d}x$ denotes the Lebesgue measure on $\Omega$ and ${\rm d}\sigma$ denotes the natural surface measure on $\Gamma$. 
It is easy to see that $\mathbb{X}^2=L^2(\Omega,{\rm d}x)\oplus L^2( \Gamma,{\rm d}\sigma)$ may be identified under the natural norm
\begin{equation*}
\|u\|_{\mathbb{X}^2}^2=\int_\Omega|u|^2{\rm d}x + \int_\Gamma|u|^2{\rm d}\sigma.
\end{equation*}
Moreover, if we identify every $u\in C(\overline{\Omega})$ with $U=(u_{\mid\Omega},u_{\mid\Gamma}) \in C(\Omega)\times C(\Gamma)$, we may also define $\mathbb{X}^2$
to be the completion of $C(\overline{\Omega})$ in the norm $\|\cdot\|_{\mathbb{X}^2}$. 
In general, any function $u\in \mathbb{X}^2$ will be of the form $u=\binom{u_1}{u_2}$ with $u_1\in L^2(\Omega,{\rm d}x)$ and $u_2\in L^2(\Gamma,{\rm d}\sigma),$ and there need not be
any connection between $u_1$ and $u_2$. 
From now on, the inner product in the Hilbert space $\mathbb{X}^2$ will be denoted by $\langle \cdot,\cdot \rangle_{\mathbb{X}^2}.$ 
Hereafter, the spaces $L^2(\Omega,{\rm d}x) $ and $L^2(\Gamma,{\rm d}\sigma)$ will simply be denoted by $L^2(\Omega) $ and $L^2(\Gamma)$.

Recall that the Dirichlet trace map ${\rm tr_D}:C^\infty (\overline{\Omega}) \rightarrow C^\infty(\Gamma),$ defined by ${\rm tr_D}(u) =u_{\mid\Gamma}$ extends to a linear continuous operator ${\rm tr_D}:H^r(\Omega)\rightarrow H^{r-1/2}(\Gamma),$ for all $r>1/2$, which is onto
for $1/2<r<3/2.$ 
This map also possesses a bounded right inverse ${\rm tr_D}^{-1}:H^{r-1/2}(\Gamma) \rightarrow H^r(\Omega)$ such that ${\rm tr_D}({\rm tr_D}^{-1}\psi) =\psi ,$ for any $\psi \in H^{r-1/2}(\Gamma) $. 
We can thus introduce the subspaces of $H^r(\Omega) \times H^{r-1/2}(\Gamma)$ and $H^r(\Omega) \times H^r(\Gamma)$, respectively, by
\begin{align}
\mathbb{V}_0^r & :=\{U=(u,\psi) \in H^r(\Omega) \times H^{r-1/2}(\Gamma) :{\mathrm{tr_{D}}}(u) =\psi \},  \notag \\
\mathbb{V}^r & :=\{U=(u,\psi) \in \mathbb{V}_0^r:{\mathrm{tr_{D}}}(u) =\psi \in H^r(\Gamma) \},  \notag
\end{align}
for every $r>1/2,$ and note that we have the following dense and compact embeddings $\mathbb{V}_0^{r_1}\subset \mathbb{V}_0^{r_2},$ for any $r_1>r_2>1/2$ (by definition, this also true for the sequence of spaces $\mathbb{V}^{r_1}\subset \mathbb{V}^{r_2}$). 
Naturally, the norm on the spaces $\mathbb{V}_0^r,$ $\mathbb{V}^r$ are defined by
\begin{equation}
\|U\|_{\mathbb{V}_0^r}^2 := \|u\|_{H^r(\Omega)}^2 + \|\psi\|_{H^{r-1/2}(\Gamma )}^2, \quad \|U\|_{\mathbb{V}^r}^2 := \|u\|_{H^r(\Omega)}^2 + \|\psi\|_{H^r(\Gamma )}^2. \notag
\end{equation}
In the sequel we are interested in the following equivalent norm in $\mathbb{V}^1$
\begin{equation*}
\|u\|_{\mathbb{V}^1}^2:=\int_\Omega \left( |\nabla u|^2 + \alpha|u|^2 \right) {\rm d}x + \int_\Gamma \left( |\nabla_\Gamma u|^2 + \beta|u|^2 \right) {\rm d}\sigma.  \label{v1b}
\end{equation*}
Naturally, the norm on the space $\mathbb{V}^r$ is defined as
\begin{equation*} \label{Vr-norm}
\|u\|^2_{\mathbb{V}^r} := \|u\|^2_{H^r(\Omega)} + \|u\|^2_{H^r(\Gamma)}.
\end{equation*}
For $U=(u,u_{\mid\Gamma})^{\rm tr}\in\mathbb{V}^1$, let $C_\Omega>0$
denote the best constant in which the Sobolev--Poincar\'e inequality holds 
\begin{equation*}  \label{Poincare}
\|u-\langle u \rangle_\Gamma \|_{L^s(\Omega)} \le C_\Omega \| \nabla u\|_{L^s(\Omega)},
\end{equation*}
for $s\geq 1$ (see \cite[Lemma 3.1]{RBT01}).
Here 
\[
\langle u \rangle_\Gamma := \frac{1}{|\Gamma|}\int_\Gamma u_{\mid\Gamma} {\rm d}\sigma.
\]

Let us now introduce the spaces for the memory variable $\eta $. 
For a nonnegative, not identically equal to zero and measurable function $\theta_S$, $S\in\{\Omega,\Gamma\}$, defined on $\mathbb{R}_+$, and a real Hilbert space $W$ (with inner product denoted by $\langle \cdot,\cdot \rangle_W$), let $L_{\theta_S}^2(\mathbb{R}_+;W)$ be the Hilbert space of $W$-valued functions on $\mathbb{R}_+$, endowed with the following inner product
\begin{equation*}  \label{sc-2}
\langle \phi_1,\phi_2\rangle_{L_{\theta_S}^2(\mathbb{R}_+;W)} := \int_0^\infty \theta_S(s) \langle \phi_1(s),\phi_2(s) \rangle _W {\rm d}s.
\end{equation*}
Moreover, for each $r>1/2$ we define
\[
L^2_{\theta_\Omega\oplus\theta_\Gamma}(\mathbb{R}_+;\mathbb{V}^r_0)\simeq L^2_{\theta_\Omega}(\mathbb{R}_+;\mathbb{V}^r)\oplus L^2_{\theta_\Gamma}(\mathbb{R}_+;H^r(\Gamma))
\]
as the Hilbert space of $\mathbb{V}^r$-valued functions $(\eta,\xi)^{\rm tr}$ on $\mathbb{R}_+$ endowed with the inner product
\begin{align*}
& \left\langle \binom{\eta_1}{\xi_1},\binom{\eta_2}{\xi_2} \right\rangle_{L^2_{\theta_\Omega\oplus\theta_\Gamma}(\mathbb{R}_+;\mathbb{V}^r)}  \notag \\
&=\int_0^\infty (\theta_\Omega(s) \langle \eta_1(s),\eta_2(s) \rangle_{H^r(\Omega)} + \theta_\Gamma(s) \langle \xi_1(s),\xi_2(s) \rangle_{H^r(\Gamma)}) {\rm d}s. \label{sc}
\end{align*}

Consequently, for $r>1/2$ we set
\begin{equation*}
\mathcal{M}_\Omega^0 := L_{\mu_\Omega}^2(\mathbb{R}_+;L^2(\Omega)), \quad \mathcal{M}_\Omega^r := L_{\mu_\Omega}^2(\mathbb{R}_+;\mathbb{V}_0^r), \quad \mathcal{M}_\Gamma^r := L_{\mu_\Gamma}^2(\mathbb{R}_+;H^r(\Gamma))
\end{equation*}
and
\begin{equation*}
\mathcal{M}_{\Omega,\Gamma}^0 := L_{\mu_\Omega\oplus \mu_\Gamma
}^2(\mathbb{R}_+;\mathbb{X}^2), \quad \mathcal{M}_{\Omega,\Gamma}^r := L_{\mu _\Omega\oplus \mu_\Gamma}^2( \mathbb{R}_+;\mathbb{V}^r).
\end{equation*}
Clearly, because of the topological identification $H^r(\Omega) \simeq \mathbb{V}_0^r$, one has the inclusion $\mathcal{M}_{\Omega,\Gamma}^r\subset \mathcal{M}_\Omega^r$ for each $r>1/2$.
We will also consider Hilbert spaces of the form $W_{\mu_\Omega}^{k,2}(\mathbb{R}_+;\mathbb{V}_0^r)$ for $k\in \mathbb{N}$. 
We also set for a matter of convenience, the inner product in $\mathcal{M}
_{\Omega,\Gamma}^1,$ as follows
\begin{align*}
& \left\langle \binom{\eta_1}{\xi_1},\binom{\eta_2}{\xi_2}\right\rangle_{\mathcal{M}
_{\Omega,\Gamma}^1} \\
& =\omega \int_0^\infty \mu_\Omega(s)\left( \langle \nabla\eta_1(s),\nabla \eta_2(s)\rangle_{L^2(\Omega)} + \alpha \langle \eta _1(s),\eta_2(s) \rangle_{L^2(\Omega)} \right) {\rm d}s \\
& + \nu \int_0^\infty \mu_\Gamma(s) \left( \langle \nabla_\Gamma\xi _1(s), \nabla_\Gamma\xi_2(s)\rangle_{L^2(\Gamma)} + \beta \langle \xi_1(s),\xi _2(s) \rangle_{L^2(\Gamma)} \right) {\rm d}s.
\end{align*}
When it is convenient, we  also use the notation
\begin{equation}
\mathcal{H}_{\Omega,\Gamma}^{0,1} := \mathbb{X}^2\times \mathcal{M}_{\Omega,\Gamma}^1, \quad \mathcal{H}_{\Omega,\Gamma}^{s,r} := \mathbb{V}^s\times \mathcal{M}_{\Omega,\Gamma}^r \quad \text{for}\ s,r\geq 1.  \notag
\end{equation}
Each space is equipped with the corresponding ``graph norm,'' whose square is defined by, for all $(U,\Phi)\in\mathcal{H}^{i,i+1}_{\Omega,\Gamma}$, $i=0,1,$ 
\begin{equation*}
\left\|(U,\Phi)\right\|^2_{\mathcal{H}^{0,1}_{\Omega,\Gamma}}:=\left\|U\right\|^2_{\mathbb{X}^2}+\left\|\Phi\right\|^2_{\mathcal{M}^1_{\Omega,\Gamma}} \quad \text{and} \quad \left\|(U,\Phi)\right\|^2_{\mathcal{H}^{1,2}_{\Omega,\Gamma}}:=\left\|U\right\|^2_{\mathbb{V}^1}+\left\|\Phi\right\|^2_{\mathcal{M}^2_{\Omega,\Gamma}}.
\end{equation*}

For the kernels $\mu_S$, $S\in\{\Omega,\Gamma\}$, we take the following
assumptions (cf. e.g. \cite{CPS06,GPM98,GPM00}). 
Assume 
\begin{align}
& \mu_S\in C^1(\mathbb{R}_+)\cap L^1(\mathbb{R}_+),  \label{mu-1} \\
& \mu_S(s) \ge 0 \quad\forall s\ge 0,  \label{mu-2} \\
& \mu'_S(s) \le 0 \quad\forall s\ge 0,  \label{mu-3} \\
& \mu'_S(s) + \delta\mu_S(s) \le 0 \quad\forall s\ge 0\ \text{and some}\ \delta>0.  \label{mu-4}
\end{align}
The assumptions (\ref{mu-1})-(\ref{mu-3}) are equivalent to assuming $k_S(s)$ be bounded, positive, nonincreasing, convex functions of class $\mathcal{C}^2$.
Moreover, assumption (\ref{mu-4}) guarantees exponential decay of the functions $\mu_S(s)$ while allowing a singularity at $s=0$. 
Assumptions (\ref{mu-1})-(\ref{mu-3}) are used in the literature (see for example \cite{CDGP-2010,CPS06,GPM98,Grasselli&Pata02-2}) to establish the existence and uniqueness of continuous global weak solutions to a system of equations similar to (\ref{problemp-1}), (\ref{problemp-3}), but with Dirichlet boundary conditions. 
In the literature (as well as here), assumption (\ref{mu-4}) is used to obtain a bounded absorbing set for the associated semigroup of solution operators.

Moving on, define 
\begin{equation*} \label{memory-4}
{\rm D(T_r)}=\left\{ \Phi\in\mathcal{M}^1_{\Omega,\Gamma}:\partial_s\Phi\in\mathcal{M}^1_{\Omega,\Gamma}, \Phi(0)=0 \right\}
\end{equation*}
where (with an abuse of notation) $\partial_s\Phi$ is the distributional derivative of $\Phi$ and the equality $\Phi(0)=0$ is meant in the following sense
\begin{equation*}
\lim_{s\rightarrow0}\|\Phi(s)\|_{\mathbb{X}^2} = 0.
\end{equation*}
Then define the linear (unbounded) operator ${\rm T_r}: {\rm D(T_r)}\rightarrow \mathcal{M}^1_{\Omega,\Gamma}$ by, for all $\Phi\in {\rm D(T_r)}$, 
\begin{equation*}
{\rm T_r}\Phi=-\frac{d}{ds}\Phi.
\end{equation*}
For each $t\in[0,T]$, the equation 
\begin{equation}  \label{memory-2}
\partial_t\Phi^t = {\rm T_r}\Phi^t + U(t) 
\end{equation}
holds as an ODE in $\mathcal{M}^1_{\Omega,\Gamma}$ subject to the initial condition
\begin{equation} \label{memory-3}
\Phi^0=\Phi_0\in\mathcal{M}^1_{\Omega,\Gamma}.
\end{equation}
The following proposition is a well-known concerning the solution to the IVP (\ref{memory-2})-(\ref{memory-3}) (this is a generalization of \cite[Theorem 3.1]{Grasselli&Pata02-2}).

\begin{proposition} \label{t:generator-T}
The operator ${\rm T_r}$ with domain ${\rm D(T_r)}$ is an infinitesimal generator of a strongly continuous semigroup of contractions on $\mathcal{M}^1_{\Omega,\Gamma}$, denoted $e^{{\rm T_r}t}$.
\end{proposition}

We now have (cf. e.g. \cite[Corollary IV.2.2]{Pazy83}).

\begin{corollary} \label{t:memory-regularity-1}
When $U\in L^1([0,T];\mathbb{V}^1)$ for each $T>0$, then, for every $\Phi_0\in\mathcal{M}^1_{\Omega,\Gamma}$, the Cauchy problem 
\begin{equation*} \label{memory-1}
\left\{ \begin{array}{ll}
\partial_t\Phi^t={\rm T_r}\Phi^t+U(t), & \text{for}\ t>0, \\ 
\Phi^0=\Phi_0, & 
\end{array} \right. 
\end{equation*}
has a unique solution $\Phi\in C([0,T];\mathcal{M}^1_{\Omega,\Gamma})$ which can be explicitly given as (cf. \cite[Section 3.2]{CPS06} and \cite[Section 3]{Grasselli&Pata02-2}) 
\begin{equation*}  \label{representation-formula-1}
\Phi^t(s)=\left\{ 
\begin{array}{ll}
\displaystyle\int_0^s U(t-y) dy, & \text{for}\ 0<s\leq t, \\ 
\displaystyle\Phi_0(s-t) + \int_0^t U(t-y) dy, & \text{when}\ s>t.
\end{array}
\right.
\end{equation*}
\end{corollary}
The interested reader can also see \cite[Section 3]{CPS06}, \cite[pp. 346--347]{GPM98} and \cite[Section 3]{Grasselli&Pata02-2} for more details concerning the above corollary in the case of static boundary conditions.
Furthermore, we also know that ${\rm T_r}$ is the infinitesimal generator of a strongly continuous (the right-translation) semigroup of contractions on $\mathcal{M}^1_{\Omega,\Gamma}$ satisfying (\ref{operator-T-1}) below; in particular, ${\rm Range}({\rm I}-{\rm T_r})=\mathcal{M}^1_{\Omega,\Gamma}$.
Following (\ref{mu-3}), there is the useful inequality.
(Also see \cite[see equation (3.4)]{CPS06} and \cite[Section 3, proof of Theorem]{Grasselli&Pata02-2}.)

\begin{corollary} \label{t:operator-T-1} 
There holds, for all $\Phi\in {\rm D(T_r)}$, 
\begin{equation} \label{operator-T-1}
\langle {\rm T_r}\Phi,\Phi \rangle_{\mathcal{M}^1_{\Omega,\Gamma}} \leq -\frac{\delta}{2}\|\Phi\|^2_{\mathcal{M}^1_{\Omega,\Gamma}}.
\end{equation}
\end{corollary}

A word of caution: even though the embedding $\mathbb{V}^1\hookrightarrow\mathbb{X}^2$ is compact, it does not follow that the embedding $\mathcal{M}^1_{\Omega,\Gamma}\hookrightarrow \mathcal{M}^0_{\Omega,\Gamma}$ is also compact. 
Indeed, see \cite{Pata-Zucchi-2001} for a counterexample.
Moreover, this means the embedding $\mathcal{H}^{1,1}_{\Omega,\Gamma}\hookrightarrow\mathcal{H}^{0,0}_{\Omega,\Gamma}$ is not compact.
Such compactness between the ``natural phase spaces'' is essential to the construction of finite dimensional exponential attractors. 
However, we treat this lack of compactness issue by following \cite{CPS06,GMPZ10} and define the so-called ``tail spaces,''
\begin{equation*}
\mathcal{K}^r_{\Omega,\Gamma} := \left\{ \Phi\in\mathcal{M}^r_{\Omega,\Gamma} : \partial_s\Phi\in\mathcal{M}^{0}_{\Omega,\Gamma},\ \Phi(0)=0,\ \sup_{\tau\ge1} \tau\mathbb{T}(\tau;\Phi)<\infty \right\} \quad \text{for}\ r\ge1,
\end{equation*}
where $\mathbb{T}(\tau;\Phi)$ is the {\em{tail function}} of $\Phi=(\eta,\xi)^{\rm tr}$ given by, for all $\tau\ge0,$ 
\begin{equation*}
\mathbb{T}(\tau;\Phi) := \int\limits_{(0,1/\tau)\cup(\tau,\infty)} \left( \mu_\Omega(s) \|\eta(s)\|^2_{L^2(\Omega)} + \mu_\Gamma \|\xi(s)\|^2_{L^2(\Gamma)} \right) {\rm d}s,
\end{equation*}
The space $\mathcal{K}^r_{\Omega,\Gamma}$ is Banach with the norm whose square is defined by
\begin{equation*}
\|\Phi\|^2_{\mathcal{K}^r_{\Omega,\Gamma}} := \|\Phi\|^2_{\mathcal{M}^r_{\Omega,\Gamma}} + \|\partial_s\Phi\|^2_{\mathcal{M}^1_{\Omega,\Gamma}} + \sup_{\tau\ge1} \tau\mathbb{T}(\tau;\Phi).  \label{new-norm}
\end{equation*}
Importantly, the embedding $\mathcal{K}^r_{\Omega,\Gamma}\hookrightarrow\mathcal{M}^{r-1}_{\Omega,\Gamma}$ is compact for any $r\ge1$ (cf. \cite[Proposition 5.4]{GMPZ10}).
Hence, let us now also define the spaces
\begin{align}
\widehat{\mathcal{H}}^{0,1}_{\Omega,\Gamma} & := \mathbb{X}^2\times\mathcal{K}^1_{\Omega,\Gamma}, \quad \widehat{\mathcal{H}}_{\Omega,\Gamma}^{s,r} := \mathbb{V}^s\times \mathcal{K}_{\Omega,\Gamma}^r \quad \text{for}\ s,r\geq 1.  \notag
\end{align}
With these spaces the desired compact embedding $\widehat{\mathcal{H}}^{1,1}_{\Omega,\Gamma} \hookrightarrow \widehat{\mathcal{H}}^{0,0}_{\Omega,\Gamma}$ holds.
Again, each space is equipped with the corresponding graph norm whose square is defined by, for all $\varepsilon\in[0,1]$ and $(U,\Phi)\in\widehat{\mathcal{H}}^{0,1}_{\Omega,\Gamma}$,
\begin{equation*}
\left\|(U,\Phi)\right\|^2_{\widehat{\mathcal{H}}^{0,1}_{\Omega,\Gamma}}:=\left\|U\right\|^2_{\mathbb{X}^2}+\left\|\Phi\right\|^2_{\mathcal{K}^1_{\Omega,\Gamma}}.
\end{equation*}
Concerning the framework for Problem {\textbf P}, an important target space that we use to apply these desired compactness properties in is the weak topology of the weak energy phase space $\mathcal{H}^{0,1}_{\Omega,\Gamma}$ which is given by 
\begin{equation}  \label{spaces}
\mathcal{H}^{-1,0}_{\Omega,\Gamma} := \mathbb{V}^{-1}\times\mathcal{M}^{0}_{\Omega,\Gamma},
\end{equation}
endowed with the canonical norm.
There holds 
\begin{equation*}  
\widehat{\mathcal{H}}^{0,1}_{\Omega,\Gamma}=\left(\mathbb{X}^2\times\mathcal{K}^1_{\Omega,\Gamma}\right) \hookrightarrow \mathcal{H}^{-1,0}_{\Omega,\Gamma},
\end{equation*}
with continuous and compact injection (indeed, the embedding $\mathbb{V}^1\hookrightarrow\mathbb{X}^2$ is compact, and recall from above the embedding $\mathcal{K}^1_{\Omega,\Gamma}\hookrightarrow\mathcal{M}^0_{\Omega,\Gamma}$ is compact).
The space $\mathcal{H}^{-1,0}_{\Omega,\Gamma}$ is the space the compact exponential attractors reside in.

The following results are \cite[Lemmas 3.3, 3.4, and 3.6]{CPS06}. 
Each presented here is an adaption to suit the framework in this article.

\begin{lemma}  \label{what-1}
Let $\Phi_0\in {\rm D(T_r)}$.
Assume there is $K>0$ such that, for all $t\ge0$, $\|U(t)\|_{\mathbb{V}^1}\le K$.
Then for all $t\ge0$,
\begin{align}
\|{\rm T_r}\Phi^t\|^2_{\mathcal{M}^1_{\Omega,\Gamma}} \le e^{-\min\{\delta_\Omega,\delta_\Gamma\} t}\|{\rm T_r}\Phi_0\|^2_{\mathcal{M}^1_{\Omega,\Gamma}} + K^2 \left( \|\mu_\Omega\|_{L^1(\mathbb{R}_+)}+\|\mu_\Gamma\|_{L^1(\mathbb{R}_+)} \right).  \notag
\end{align}
\end{lemma}

\begin{lemma}  \label{what-2}
Let $\Phi_0\in {\rm D(T_r)}$.
Assume there is $K>0$ such that, for all $t\ge0$, $\|U(t)\|_{\mathbb{V}^1}\le K$.
Then there is a constant $C>0$ such that, for all $t\ge0$,
\begin{align}
\sup_{\tau\ge1} \tau\mathbb{T}(\tau;\Phi^t) \le 2 \left( t+2 \right)e^{-\min\{\delta_\Omega,\delta_\Gamma\} t} \sup_{\tau\ge1} \tau\mathbb{T}(\tau;\Phi_0) + CK^2.  \notag
\end{align}
\end{lemma}

Next, we consider the linear (self-adjoint, positive) operator ${\rm B}\psi :={\rm B}_\beta\psi =-\Delta _\Gamma\psi +\beta \psi $ acting on ${\rm D(B)} =H^2(\Gamma) $. 
The basic (linear) operator associated with problem (\ref{eq1m})-(\ref{eq2m}) is the so-called ``Wentzell'' Laplace operator.
Recall that $\omega \in( 0,1) $. 
We let
\begin{align*}
{\rm A_W^{\alpha,\beta,\nu,\omega}}\binom{u_1}{u_2} & :=\left( 
\begin{array}{cc}
-\omega \Delta +\alpha \omega {\rm I} & 0 \\ 
\omega \partial_{\bf n}(\cdot)  & \nu {\rm B}
\end{array}
\right) \left( 
\begin{array}{c}
u_1 \\ 
u_2
\end{array}
\right)   \notag \\
& ={\rm A_W^{\alpha,0,0,\omega}}\binom{u_1}{u_2}+\binom{0}{\nu {\rm B}u_2},  \label{A_Wentzell1}
\end{align*}
with
\begin{equation*}
{\rm D(A_W^{\alpha,\beta,\nu,\omega})} :=\left\{ U=\binom{u_1}{u_2}\in \mathbb{Y}:-\Delta u_1\in L^2(\Omega),\ \omega \partial_{\bf n}u_1-\nu {\rm B}u_2\in L^2(\Gamma) \right\},  \label{A_Wentzell2}
\end{equation*}
where $\mathbb{Y}:=\mathbb{V}_0^1$ if $\nu =0,$ and $\mathbb{Y}:=\mathbb{V}^1$ if $\nu >0.$ 
It is well-known that $({\rm A_W^{\alpha,\beta,\nu,\omega}},{\rm D(A_W^{\alpha,\beta,\nu,\omega})})$ is self-adjoint and nonnegative operator on $\mathbb{X}^2$ whenever $\alpha,\beta,\nu\ge 0,$ and ${\rm A_W^{\alpha,\beta,\nu,\omega}}>0$ if either $\alpha>0$ or $\beta >0$. 
Moreover, the resolvent operator $({\rm I}+{\rm A_W^{\alpha,\beta,\nu,\omega}})^{-1}\in \mathcal{L}(\mathbb{X}^2) $ is compact. 
Moreover, since $\Gamma$ is of class $\mathcal{C}^2,$ then ${\rm D(A_W^{\alpha,\beta,\nu,\omega})}=\mathbb{V}^2$ if $\nu>0$. 
Indeed, for any $\alpha,\beta \ge 0$ with $(\alpha,\beta) \neq (0,0),$ the map $\Psi
:U\mapsto {\rm A_W^{\alpha,\beta,\nu,\omega}}U,$ when viewed as a map from $\mathbb{V}^2$ into $\mathbb{X}^2=L^2(\Omega)\times L^2(\Gamma),$ is an isomorphism, and there exists a
positive constant $C_\ast$, independent of $U=(u,\psi)^{\rm tr}$, such that
\begin{equation*}
C_\ast^{-1}\|U\|_{\mathbb{V}^2} \le \|\Psi(U)\|_{\mathbb{X}^2} \le C_\ast\|U\|_{\mathbb{V}^2},  \label{regularity-oper}
\end{equation*}
for all $U\in \mathbb{V}^2$ (cf. Lemma \ref{t:appendix-lemma-3}). 
Whenever $\nu=0$, by elliptic regularity theory and $U\in {\rm D(A_W^{\alpha,\beta,0,\omega})}$ one has $u\in H^{3/2}(\Omega) $ and $\psi={\rm tr_D}(u) \in H^1(\Gamma) $, since
the Dirichlet-to-Neumann map is bounded from $H^1(\Gamma)$ to $L^2(\Gamma)$; hence ${\rm D(A_W^{\alpha,\beta,0,\omega})}=\mathbb{W}$, where $\mathbb{W}$ is the Hilbert space equipped with the following (equivalent) norm
\begin{equation*}
\|U\|_{\mathbb{W}}^2 :=\|U\|_{\mathbb{V}_0^{3/2}}^2 + \|\Delta u\|_{L^2(\Omega)}^2 + \|\partial_{\bf n}u\|_{L^2(\Gamma)}^2.
\end{equation*}
Concerning the ``Wentzell'' Laplacian we refer the reader to more details to e.g., \cite{CGGM10,CFGGGOR09,Gal&Warma10}, and the references therein. We now have all
the necessary ingredients to introduce a rigorous formulation of Problem 
\textbf{P} in the next section.

For the nonlinear terms, first we assume $f,g \in C^0(\mathbb{R})$ satisfy the growth assumptions for some positive constants $\ell_1$ and $\ell_2$, and $d\ge2$ such that for all $r,s\in \mathbb{R}$, 
\begin{align}
& |f(r)-f(s)| \leq \ell_1(1+|r-s|^2),  \label{assm-1} \\
& |g(r)-g(s)| \leq \ell_2(1+|r-s|^{d-1}).  \label{assm-2}
\end{align}
Below we will set $F:\mathbb{R}^2\rightarrow\mathbb{R}^2,$
\begin{equation}
F(U):=\begin{pmatrix}f(u) \\ \widetilde{g}(u)\end{pmatrix},  \label{func}
\end{equation}
where $\widetilde{g}(s):=g(s)-\omega\beta s$, for $s\in\mathbb{R}.$
(To offset $\widetilde{g}$, the term $\omega\beta u$ will be incorporated in the operator ${\rm A_W^{0,0,\nu,0}}$ as ${\rm A_W^{0,\beta,\nu,\omega}}.$)
The next assumptions are that there are positive constants, $\kappa_i,$ $i=1,2,3,4,$ so that, for all $s\in\mathbb{R},$
\begin{align}
& f(s)s \ge \kappa_1|s|^4-\kappa_2,  \label{assm-3} \\
& \widetilde g(s)s \ge \kappa_3|s|^{r}-\kappa_4.  \label{assm-4}
\end{align}

The above conditions (\ref{assm-1})-(\ref{assm-4}) are suitable to prove the existence of an absorbing set for the family of solution operators generated by the weak solutions of Problem \textbf{P}.
When we consider the quasi-strong solutions we will employ the following conditions.
For some positive constants $\ell_1$ and $\ell_2$, and $d\ge2$ such that for all $r,s\in \mathbb{R}$, 

\begin{align}
& |f'(s)|\le \ell_1(1+|s|^2), \label{quasi-assm-1} \\
& |g'(s)|\le \ell_2(1+|s|^d). \label{quasi-assm-2}
\end{align}
In addition, in this case we will assume there are $M_f,M_g>0$ such that, for all $s\in\mathbb{R}$,
\begin{align}
f'(s) \ge -M_f, \label{quasi-assm-3} \\
g'(s) \ge -M_g. \label{quasi-assm-4}
\end{align}
Let
\begin{equation*}
h_f(s) =\int_0^sf'(\tau) \tau {\rm d}\tau \quad \text{and} \quad h_g(s) =\int_0^s\widetilde{g}'(\tau) \tau {\rm d}\tau.
\end{equation*}
Finally, in this case we will assume there exist $C_i>0,$ $i=1,\dots,8,$ such that, for all $s\in\mathbb{R}$,
\begin{align}
f(s)s\ge -C_1|s|^2-C_2, \label{quasi-assm-5} \\
g(s)s\ge -C_3|s|^2-C_4, \label{quasi-assm-6} \\
h_f(s) \ge -C_5|s|^2-C_6, \label{quasi-assm-7} \\
h_g(s) \ge -C_7|s|^2-C_8. \label{quasi-assm-8}
\end{align}

\begin{remark}
Observe that here we do not allow for the critical polynomial growth exponent (of power $5$ in (\ref{assm-1})) which appears in several works with static boundary conditions (cf. e.g. \cite{CDGP-2010,CPS06}). 
\end{remark}

Constants appearing below may depend on various structural parameters such as $\alpha$, $\beta,$ $\delta$, $\nu$, $\omega$, $|\Omega|$, $|\Gamma|$, $\ell_f$, $\ell_g$ and $d$, as well as those constants appearing in (\ref{assm-3})-(\ref{quasi-assm-8}), and the constants may even change from line to line.
We denote by $Q(\cdot)$ a generic monotonically increasing function. 
We will use $\|B\|_W:=\sup_{\Upsilon\in B}\|\Upsilon\|_W$ to denote the ``size'' of the subset $B$ in the Banach space $W$.

\section{Review of well-posedness}

Here we provide some definitions and cite the relevant global well-posedness results concerning Problem {\textbf{P}}.
For the remainder of this article we choose to set $n=3$, which is of course the most relevant physical dimension.

\begin{definition}  \label{d:weak-solution} 
Let $\alpha,\beta>0$, $\nu,\omega\in(0,1)$ and $T>0$. 
Given $U_0=(u_0,v_0)^{\rm tr}\in\mathbb{X}^2$ and $\Phi_0=(\eta_0,\xi_0)^{\rm tr}\in\mathcal{M}^1_{\Omega,\Gamma}$, the pair $U(t)=(u(t),v(t))^{\rm tr}$ and $\Phi^t=(\eta^t,\xi^t)^{\rm tr}$ satisfying
\begin{align*}
U &\in L^{\infty }(0,T;\mathbb{X}^2)\cap L^2(0,T;\mathbb{V}^1),  \\
u &\in L^4(\Omega\times(0,T)),  \\ 
v &\in L^{r}(\Gamma\times(0,T)),  \\ 
\Phi &\in L^{\infty }\left(0,T;\mathcal{M}^1_{\Omega,\Gamma}\right),  \\ 
\partial_t U &\in L^2\left(0,T;(\mathbb{V}^1)^*\right) \oplus \left( L^{4/3}(\Omega\times(0,T)) \times L^{r'}(\Gamma\times(0,T)) \right),  \\ 
\partial_t \Phi &\in L^2\left(0,T;W^{-1,2}_{\mu_\Omega\oplus\mu_\Gamma}(\mathbb{R}_+;\mathbb{V}^1) \right), 
\end{align*}
is said to be a weak solution to Problem {\textbf{P}} if, $v(t)=u_{\mid\Gamma}(t)$ and $\xi^t=\eta^t_{\mid\Gamma}$ for almost all $t\in(0,T]$, and
for all $\Xi =(\varsigma ,\varsigma _{\mid \Gamma })^{\rm tr}\in \mathbb{V}^1 \cap \left( L^2(\Omega) \times L^r(\Gamma) \right)$, $\Pi =(\rho ,\rho _{\mid \Gamma })^{\rm tr}\in \mathcal{M}^1_{\Omega,\Gamma}$, and for almost all $t\in(0,T]$, there holds, 
\begin{align}
\left\langle \partial_tU(t),\Xi \right\rangle _{\mathbb{X}^2} & + \left\langle {{\rm A_W^{0,\beta,\nu,\omega}}}U(t),\Xi \right\rangle_{\mathbb{X}^2}+ \int_0^\infty \mu_\Omega(s)\left\langle {{\rm A_W^{\alpha,0,0,\omega}}}\Phi^t(s),\Xi \right\rangle_{\mathbb{X}^2} {\rm d}s \notag \\
& + \nu\int_0^\infty \mu_\Gamma(s)\left\langle {\rm B}\xi^t(s),\varsigma_{\mid\Gamma} \right\rangle_{L^2(\Gamma)} {\rm d}s + \left\langle F(U(t)),\Xi\right\rangle_{\mathbb{X}^2}=0,  \label{weak-solution-1}
\end{align}
\begin{equation}  \label{weak-solution-2}
\left\langle \partial_t\eta^t,\rho \right\rangle _{\mathcal{M}^1_\Omega}=\left\langle {\rm T_r}\eta^t,\rho \right\rangle _{\mathcal{M}^1_\Omega}+\left\langle u(t),\rho \right\rangle _{\mathcal{M}^1_\Omega},  
\end{equation}
\begin{equation}  \label{weak-solution-2.5}
\left\langle \partial_t\xi^t,\rho_{\mid\Gamma} \right\rangle _{\mathcal{M}^1_\Gamma}=\left\langle {\rm T_r}\xi^t,\rho_{\mid\Gamma} \right\rangle_{\mathcal{M}^1_\Gamma}+\left\langle v(t),\rho_{\mid\Gamma} \right\rangle_{\mathcal{M}^1_\Gamma},  
\end{equation}
in addition, 
\begin{equation}
U(0)=U_0 \quad\text{and}\quad \Phi^0=\Phi_0. \label{weak-solution-3}
\end{equation}
The function $[0,T]\ni t\mapsto (U(t),\Phi^t )$ is called a global weak solution if it is a weak solution for every $T>0$.
\end{definition}

\begin{remark}  \label{r:trace-map}
When we have a weak solution to Problem {\textbf{P}}, the above restrictions $u_{\mid\Gamma}(t)$ and $\eta_{\mid\Gamma}^t$ are well-defined by virtue of the Dirichlet trace map, ${\rm tr_D}:H^1(\Omega)\rightarrow H^{1/2}(\Gamma)$. However, this is not necessarily the case for $\partial_t U$.
\end{remark}

\begin{definition}  \label{d:strong-solution} 
The pair $U(t)=(u(t),v(t))^{\rm tr}$ and $\Phi^t =(\eta^t,\xi^t)^{\rm tr}$ is called a quasi-strong solution of Problem {\textbf P} on $[0,T)$ if $(U(t),\Phi^t)$ satisfies the equations (\ref{weak-solution-1})-(\ref{weak-solution-3}) for all $\Xi \in \mathbb{V}^1$, $\Pi \in \mathcal{M}^1_{\Omega,\Gamma}$, almost everywhere on $(0,T)$ and if it has the regularity properties:
\begin{align*}
U &\in L^\infty \left( 0,T;\mathbb{V}^1 \right) \cap W^{1,2} \left( 0,T;\mathbb{V}^1 \right),   \\
\Phi &\in L^\infty \left( 0,T;{\rm D(T_r)} \right),   \\
\partial_tU &\in L^\infty\left( 0,T;\mathbb{X}^2 \right),   \\
\partial_t\Phi &\in L^\infty\left( 0,T;\mathcal{M}^1_{\Omega,\Gamma}\right).  
\end{align*}
As before, the function $[0,T]\ni t\mapsto (U(t),\Phi^t)$ is called a global quasi-strong solution if it is a quasi-strong solution for every $T>0$.
\end{definition}

The following is \cite[Theorem 3.9]{Gal-Shomberg15-2}.

\begin{theorem}   \label{t:weak-solutions}
Assume $\mu_S$, $S\in\{\Omega,\Gamma\}$, satisfy (\ref{mu-1})-(\ref{mu-3}) and $f,g\in C^0(\mathbb{R})$ satisfy (\ref{assm-1})-(\ref{assm-2}) and (\ref{assm-3})-(\ref{assm-4}). 
For each $\alpha,\beta>0$, $\omega,\nu\in (0,1)$ and $T>0$, and for any $U_0=(u_0,v_0)^{\rm tr}\in \mathbb{X}^2$, $\Phi_0=(\eta_0,\xi_0)^{\rm tr}\in \mathcal{M}_{\Omega,\Gamma }^1,$ there exists at least one global weak solution $(U,\Phi)$ to Problem {\textbf P} in the sense of Definition \ref{d:weak-solution}.
\end{theorem}

Now we have \cite[Proposition 3.10]{Gal-Shomberg15-2}.

\begin{proposition}      \label{t:cont-dep} 
Suppose the assumptions of Theorem \ref{t:weak-solutions} hold. 
The global weak solution to Problem {\textbf{P}} is unique and depends continuously on the initial datum in the following way; there exists a constant $C>0$, independent of $U_i$, $\Phi_i$, $i=1,2$, and $T>0$ in which, for all $t\in[0,T]$, there holds
\begin{align}
\left\| U_1(t)-U_2(t) \right\|_{\mathbb{X}^2} + \left\| \Phi^t_1-\Phi^t_2 \right\|_{\mathcal{M}^1_{\Omega,\Gamma}} \le \left( \left\| U_1(0)-U_2(0) \right\|_{\mathbb{X}^2} + \left\| \Phi^0_1-\Phi^0_2 \right\|_{\mathcal{M}^1_{\Omega,\Gamma}} \right) e^{Ct}. \label{cde}
\end{align}
\end{proposition}

\begin{theorem}  \label{t:quasi-strong}
Assume $\mu_S$, $S\in\{\Omega,\Gamma\}$, satisfy (\ref{mu-1})-(\ref{mu-4}) and $f,g\in C^1(\mathbb{R})$ satisfy (\ref{quasi-assm-1})-(\ref{quasi-assm-8}). 
For each $\alpha,\beta >0$, $\omega,\nu \in (0,1)$, $T>0$ and for any $U_0=(u_0,v_0)^{\rm tr}\in\mathbb{V}^2$,$\Phi_0=(\eta_0,\xi_0)^{\rm tr} \in (\mathcal{M}^2_{\Omega,\Gamma} \cap {\rm D(T_r)})$, there exists a unique global quasi-strong solution $(U,\Phi)$ to Problem \textbf{P} in the sense of Definition \ref{d:strong-solution}.
\end{theorem}

We conclude the preliminary results for Problem {\textbf{P}} with the following. 

\begin{corollary} \label{t:sg}
Problem {\textbf{P}} defines a (nonlinear) strongly continuous semigroup $\mathcal{S}(t)$ on the phase space $\mathcal{H}^{0,1}_{\Omega,\Gamma} = \mathbb{X}^2\times\mathcal{M}^1_{\Omega,\Gamma}$ by 
\begin{equation*}
\mathcal{S}(t)\Upsilon_0 := \left( U(t),\Phi^t \right),
\end{equation*}
where $\Upsilon_0=(U_0,\Phi_0)\in \mathcal{H}^{0,1}_{\Omega,\Gamma}$ and $( U(t),\Phi^t )$ is the unique solution to Problem {\textbf{P}}. 
The semigroup is Lipschitz continuous on $\mathcal{H}^{0,1}_{\Omega,\Gamma}$ via the continuous dependence estimate (\ref{cde}).
\end{corollary}

\section{Dissipation of weak solutions}

This section is dedicated to three results for the weak solutions to Problem {\textbf P}.
We will show the existence of a bounded absorbing set in the phase space $\mathcal{H}^{0,1}_{\Omega,\Gamma}=\mathbb{X}^2\times\mathcal{M}^1_{\Omega,\Gamma}.$
Hence, we establish that the semigroup of solution operators, $\mathcal{S}(t):\mathcal{H}^{0,1}_{\Omega,\Gamma}\rightarrow\mathcal{H}^{0,1}_{\Omega,\Gamma}$, defined above in Corollary \ref{t:sg}, is dissipative. 
It was already established in \cite[Lemma 3.2]{Shomberg-reacg16} that when $\mu_\Omega=\mu_\Gamma$, the semigroup of solution operators admits a bounded absorbing set $\mathcal{B}^0$ in the weak energy phase space $\mathcal{H}^{0,1}_{\Omega,\Gamma}$.
We now show the existence of a bounded absorbing set for the general case where $\mu_\Omega\not=\mu_\Gamma$, however, our argument requires a smallness condition on $k_\Gamma(0)$ (see the proof of Lemma \ref{t:to-H2}).
Precisely, the following dissipation result assumes 
\begin{equation}  \label{assk}
k_\Gamma(0)\le\frac{4}{1-\omega}.
\end{equation}
The final result in this section concerns a local Lipschitz continuity result for the weak solutions in the weak space called $\mathcal{H}^{-1,0}_{\Omega,\Gamma}$ defined below in (\ref{spaces}).
It is important to recall that, because of the (rather strict) assumption to Lemma \ref{what-1}, that $\|U(t)\|_{\mathbb{V}^1}\le K$ for some constant $K>0$ for all $t\ge0,$ we are not able to establish that the weak solutions admit a {\em compact} absorbing set in the weak topology of $\mathcal{H}^{0,1}_{\Omega,\Gamma}$.
In order to apply Lemma \ref{what-1}, and, in turn, provide the existence of finite dimensional exponential attractors in the weak topology of $\mathcal{H}^{0,1}_{\Omega,\Gamma}$, we will need to employ the fairly smoother quasi-strong solutions defined above (see Definition \ref{d:strong-solution}).
Thus, the bounded absorbing set for the weak solutions in $\mathcal{H}^{0,1}_{\Omega,\Gamma}$ can be used to construct a {\em compact} absorbing set for the quasi-strong solutions in $\mathcal{H}^{-1,0}_{\Omega,\Gamma}$.

Our first result concerning the weak solutions of Problem {\textbf P} now follows.

\begin{lemma}  \label{t:weak-ball}
In addition to the assumptions of Theorem \ref{t:weak-solutions}, assume (\ref{mu-4}) and (\ref{assk}) hold. 
For all $R>0$ and $\Upsilon_0=(U_0,\Phi_0)\in\mathcal{H}^{0,1}_{\Omega,\Gamma}=\mathbb{X}^2\times\mathcal{M}^1_{\Omega,\Gamma}$ with $\|\Upsilon_0\|_{\mathcal{H}^{0,1}_{\Omega,\Gamma}}\le R$, there exist positive constants $c_0=c_0(\beta,\delta,\nu,\omega,k_\Gamma(0))$ and $P_0=P_0(\kappa_2,\kappa_4,c_0)$, and there is a positive monotonically increasing function $Q(\cdot)$ in which, for all $t\ge0$,
\begin{align}
\left\| \left( U(t),\Phi^t \right) \right\|^2_{\mathcal{H}^{0,1}_{\Omega,\Gamma}} \le Q(R) e^{-c_0t} + P_0.  \label{weak-decay} 
\end{align}
Also, 
\begin{align}
\int_0^t\|U(\tau)\|^2_{\mathbb{V}^1}{\rm d}\tau \le Q(R). \label{fus-11.5}
\end{align}
\end{lemma}

\begin{proof}
Let $R>0$ and $\Upsilon_0=(U_0,\Phi_0)\in\mathcal{H}^{0,1}_{\Omega,\Gamma}=\mathbb{X}^2\times\mathcal{M}^1_{\Omega,\Gamma}$ be such that $\|\Upsilon_0\|_{\mathcal{H}^{0,1}_{\Omega,\Gamma}}\le R.$
From the equations (\ref{weak-solution-1})-(\ref{weak-solution-2.5}), we take the corresponding weak solution $\Xi=U(t)=(u(t),u_{\mid\Gamma}(t))^{\rm tr}$ and $\Pi(s)=\Phi^t(s)=(\eta^t(s),\xi^t(s))^{\rm tr}$ and we obtain the identities 
\begin{align}
\left\langle \partial_tU,U \right\rangle _{\mathbb{X}^2} & + \left\langle {{\rm A_W^{0,\beta,\nu,\omega}}}U,U \right\rangle_{\mathbb{X}^2} + \int_0^\infty \mu_\Omega(s)\left\langle {{\rm A_W^{\alpha,0,0,\omega}}}\Phi^t(s),U \right\rangle_{\mathbb{X}^2} {\rm d}s \notag \\
& + \nu\int_0^\infty \mu_\Gamma(s)\left\langle {\rm B}\xi^t(s),\xi^t(s) \right\rangle_{L^2(\Gamma)} {\rm d}s + \left\langle F(U),U\right\rangle_{\mathbb{X}^2}=0,  \label{fus-1}
\end{align}
\begin{equation}  \label{fus-2}
\left\langle \partial_t\eta^t,\eta^t \right\rangle _{\mathcal{M}^1_\Omega}=\left\langle {\rm T_r}\eta^t,\eta^t \right\rangle _{\mathcal{M}^1_\Omega}+\left\langle u(t),\eta^t \right\rangle _{\mathcal{M}^1_\Omega},  
\end{equation}
and
\begin{equation}  \label{fus-2.5}
\left\langle \partial_t\xi^t,\xi^t \right\rangle _{\mathcal{M}^1_\Gamma}=\left\langle {\rm T_r}\xi^t,\xi^t \right\rangle _{\mathcal{M}^1_\Gamma}+\left\langle u(t),\xi^t \right\rangle _{\mathcal{M}^1_\Gamma}.
\end{equation}
Observe,
\begin{align}
& \left\langle \partial_t U,U \right\rangle_{\mathbb{X}^2} = \frac{1}{2}\frac{{\rm d}}{{\rm d}t}\left\| U \right\|^2_{\mathbb{X}^2},  \label{fus-3} \\
& \left\langle {{\rm A_W^{0,\beta,\nu,\omega}}}U,U \right\rangle_{\mathbb{X}^2} = \omega\|\nabla u\|^2_{L^2(\Omega)}+\nu\|\nabla_\Gamma u\|^2_{L^2(\Gamma)}+\beta\nu\|u\|^2_{L^2(\Gamma)}, \label{fus-4}
\end{align}
and
\begin{align}
\left\langle \partial_t\eta^t,\eta^t \right\rangle _{\mathcal{M}^1_\Omega} + \left\langle \partial_t\xi^t,\xi^t \right\rangle _{\mathcal{M}^1_\Gamma} = \left\langle \partial_t\Phi^t,\Phi^t \right\rangle_{\mathcal{M}^1_{\Omega,\Gamma}} = \frac{1}{2}\frac{{\rm d}}{{\rm d}t}\left\| \Phi^t \right\|^2_{\mathcal{M}^1_{\Omega,\Gamma}}.  \label{fus-5}
\end{align}
Combining (\ref{fus-1})-(\ref{fus-5}) produces the differential identity, which holds for almost all $t\ge0$,
\begin{align}
& \frac{1}{2}\frac{{\rm d}}{{\rm d}t}\left\{ \left\| U \right\|^2_{\mathbb{X}^2} + \left\| \Phi^t \right\|^2_{\mathcal{M}^1_{\Omega,\Gamma}} \right\} + \omega\|\nabla u\|^2_{L^2(\Omega)} + \nu\|\nabla_\Gamma u\|^2_{L^2(\Gamma)} + \beta\nu\|u\|^2_{L^2(\Gamma)} \notag \\ 
& + \int_0^\infty \mu_\Omega(s)\left\langle {{\rm A_W^{\alpha,0,0,\omega}}}\Phi^t(s),U \right\rangle_{\mathbb{X}^2} {\rm d}s + \nu\int_0^\infty \mu_\Gamma(s)\left\langle {\rm B}\xi^t(s),\xi^t(s) \right\rangle_{L^2(\Gamma)} {\rm d}s \notag \\
&  + \left\langle F(U),U \right\rangle_{\mathbb{X}^2} - \left\langle {\rm T_r}\Phi^t,\Phi^t \right\rangle_{\mathcal{M}^1_{\Omega,\Gamma}} - \left\langle U(t),\Phi^t \right\rangle _{\mathcal{M}^1_{\Omega,\Gamma}}= 0.  \label{fus-6}
\end{align}
Because of assumption (\ref{mu-4}),  we may directly apply (\ref{operator-T-1}) from Corollary \ref{t:operator-T-1}; i.e., 
\begin{equation}  \label{fus-7}
-\left\langle {\rm T_r}\Phi^t,\Phi^t \right\rangle_{\mathcal{M}^1_{\Omega,\Gamma}} \ge \frac{\delta}{2}\left\|\Phi^t\right\|^2_{\mathcal{M}^1_{\Omega,\Gamma}}.
\end{equation}
With (\ref{assm-3}) and (\ref{assm-4}), we know  
\begin{align}
\left\langle F(U),U \right\rangle_{\mathbb{X}^2} & = \langle f(u),u\rangle_{L^2(\Omega)}+\langle \widetilde{g}(u),u\rangle_{L^2(\Gamma)} \notag \\
& \ge \kappa_1\|u\|^4_{L^4(\Omega)} + \kappa_3\|u\|^{r}_{L^r(\Gamma)} - (\kappa_2+\kappa_4).  \label{fus-8} 
\end{align}
Hence, (\ref{fus-6})-(\ref{fus-8}) yields the differential inequality,
\begin{align}
& \frac{{\rm d}}{{\rm d}t}\left\{ \left\| U \right\|^2_{\mathbb{X}^2} + \left\| \Phi^t \right\|^2_{\mathcal{M}^1_{\Omega,\Gamma}} \right\} + 2\omega\|\nabla u\|^2_{L^2(\Omega)} + 2\nu\|\nabla_\Gamma u\|^2_{L^2(\Gamma)} + 2\beta\nu\|u\|^2_{L^2(\Gamma)} \notag \\  
& + 2\int_0^\infty \mu_\Omega(s)\left\langle {{\rm A_W^{\alpha,0,0,\omega}}}\Phi^t(s),U \right\rangle_{\mathbb{X}^2} {\rm d}s + 2\nu\int_0^\infty \mu_\Gamma(s)\left\langle {\rm B}\xi^t(s),\xi^t(s) \right\rangle_{L^2(\Gamma)} {\rm d}s \notag \\
&  - 2\left\langle U(t),\Phi^t \right\rangle _{\mathcal{M}^1_{\Omega,\Gamma}} + \delta \left\|\Phi^t\right\|^2_{\mathcal{M}^1_{\Omega,\Gamma}} + 2\kappa_1\|u\|^4_{L^4(\Omega)} + 2\kappa_3\|u\|^{r}_{L^{r}(\Gamma)} \le 2\left(\kappa_2+\kappa_4\right).  \label{fus-9}
\end{align}
Recall
\begin{align*}
2\int_0^\infty \mu_\Omega(s) & \left\langle {{\rm A_W^{\alpha,0,0,\omega}}}\Phi^t(s),U \right\rangle_{\mathbb{X}^2} {\rm d}s \\
& = 2\omega\int_0^\infty \mu_\Omega(s)\left\langle -\Delta\eta^t(s)+\alpha\eta^t(s),u \right\rangle_{L^2(\Omega)}{\rm d}s + 2\omega\int_0^\infty \mu_\Omega(s)\left\langle \partial_{\bf n} \eta^t(s),u \right\rangle_{L^2(\Gamma)} {\rm d}s \\
& = 2\omega\int_0^\infty \mu_\Omega(s)\left\langle \nabla\eta^t(s),\nabla u\right\rangle_{L^2(\Omega)} ds+2\alpha\omega\int_0^\infty\mu_\Omega(s)\left\langle\eta^t(s),u \right\rangle_{L^2(\Omega)}{\rm d}s,
\end{align*}
\begin{align*}
2\nu\int_0^\infty \mu_\Gamma(s)\left\langle {\rm B}\xi^t(s),\xi^t(s) \right\rangle_{L^2(\Gamma)} {\rm d}s & = 2\nu\int_0^\infty \mu_\Gamma(s)\left\langle -\Delta_\Gamma\xi^t(s)+\beta\xi^t(s),\xi^t(s) \right\rangle_{L^2(\Gamma)} {\rm d}s \\
& = 2\nu\int_0^\infty \mu_\Gamma(s)\left\|\nabla_\Gamma\xi^t(s)\right\|^2_{L^2(\Gamma)}ds+2\beta\nu\int_0^\infty \mu_\Gamma(s)\left\|\xi^t(s)\right\|^2_{L^2(\Gamma)}{\rm d}s,
\end{align*}
and 
\begin{align*}
- 2\left\langle U(t),\Phi^t \right\rangle _{\mathcal{M}^1_{\Omega,\Gamma}} & = -2\omega \int_0^\infty \mu_\Omega(s) \langle \nabla u,\nabla \eta^t(s)\rangle_{L^2(\Omega)} {\rm d}s - 2\alpha\omega \int_0^\infty \mu_\Omega(s)  \langle u,\eta^t(s) \rangle_{L^2(\Omega)} {\rm d}s \\
& - 2\nu \int_0^\infty \mu_\Gamma(s) \langle \nabla_\Gamma u, \nabla_\Gamma\xi^t(s)\rangle_{L^2(\Gamma)} {\rm d}s -  2\beta\nu \int_0^\infty \mu_\Gamma(s)  \langle u,\xi^t(s) \rangle_{L^2(\Gamma)} {\rm d}s.
\end{align*}
Thus,
\begin{align}
& 2\int_0^\infty \mu_\Omega(s) \left\langle {{\rm A_W^{\alpha,0,0,\omega}}}\Phi^t(s),U \right\rangle_{\mathbb{X}^2} {\rm d}s + 2\nu\int_0^\infty \mu_\Gamma(s)\left\langle {\rm B}\xi^t(s),\xi^t(s) \right\rangle_{L^2(\Gamma)} {\rm d}s - 2\left\langle U(t),\Phi^t \right\rangle _{\mathcal{M}^1_{\Omega,\Gamma}} \notag \\
& = 2\nu\int_0^\infty \mu_\Gamma(s)\left\|\nabla_\Gamma\xi^t(s)\right\|^2_{L^2(\Gamma)}ds+2\beta\nu\int_0^\infty \mu_\Gamma(s)\left\|\xi^t(s)\right\|^2_{L^2(\Gamma)}{\rm d}s \notag \\
& - 2\nu \int_0^\infty \mu_\Gamma(s) \langle \nabla_\Gamma u, \nabla_\Gamma\xi^t(s)\rangle_{L^2(\Gamma)} {\rm d}s -  2\beta\nu \int_0^\infty \mu_\Gamma(s)  \langle u,\xi^t(s) \rangle_{L^2(\Gamma)} {\rm d}s \notag \\
& \ge -\frac{\nu}{2}\int_0^\infty\mu_\Gamma(s)\left\|\nabla_\Gamma u\right\|^2_{L^2(\Gamma)}{\rm d}s -\frac{\beta\nu}{2}\int_0^\infty\mu_\Gamma(s)\left\|u\right\|^2_{L^2(\Gamma)}{\rm d}s \notag \\
& = -m_\Gamma\frac{\nu}{2} \left\|\nabla_\Gamma u\right\|^2_{L^2(\Gamma)} - m_\Gamma\frac{\beta\nu}{2} \left\|u\right\|^2_{L^2(\Gamma)}, \label{fus-9.1}
\end{align}
where we defined
\begin{align*}
m_\Gamma & =\int_0^\infty\mu_\Gamma(s){\rm d}s \\
& = -(1-\omega)\int_0^\infty k'_\Gamma(s){\rm d}s \\
& = -(1-\omega)\left( \lim_{\sigma\rightarrow+\infty}k_\Gamma(\sigma)-k_\Gamma(0) \right) \\
& = (1-\omega)k_\Gamma(0).
\end{align*}
Together (\ref{fus-9}) and (\ref{fus-9.1}) yield
\begin{align}
& \frac{{\rm d}}{{\rm d}t}\left\{ \left\| U \right\|^2_{\mathbb{X}^2} + \left\| \Phi^t \right\|^2_{\mathcal{M}^1_{\Omega,\Gamma}} \right\} + 2\omega\|\nabla u\|^2_{L^2(\Omega)} + \nu\left(2-\frac{m_\Gamma}{2}\right)\|\nabla_\Gamma u\|^2_{L^2(\Gamma)}  \notag \\
& + \beta\nu\left(2-\frac{m_\Gamma}{2}\right) \|u\|^2_{L^2(\Gamma)} + \delta \left\|\Phi^t\right\|^2_{\mathcal{M}^1_{\Omega,\Gamma}} + 2\kappa_1\|u\|^4_{L^4(\Omega)} + 2\kappa_3\|u\|^{r}_{L^{r}(\Gamma)} \notag \\
& \le 2\left(\kappa_2+\kappa_4\right).  \label{fus-9.2}
\end{align}
Thanks to assumption (\ref{assk}), we set 
\[
c_0:=\min\left\{2\omega,\beta\nu\left(2-\frac{m_\Gamma}{2}\right),\delta\right\}>0,
\]
so that we can write (\ref{fus-9.2}) as the differential inequality, which holds for almost all $t\ge0,$
\begin{align}
& \frac{{\rm d}}{{\rm d}t}\left\{ \left\| U \right\|^2_{\mathbb{X}^2} + \left\| \Phi^t \right\|^2_{\mathcal{M}^1_{\Omega,\Gamma}} \right\} + c_0 \left( \left\| U \right\|^2_{\mathbb{V}^1} + \left\|\Phi^t\right\|^2_{\mathcal{M}^1_{\Omega,\Gamma}} \right) \notag \\
&  + 2\kappa_1\|u\|^4_{L^4(\Omega)} + 2\kappa_3\|u\|^{r}_{L^{r}(\Gamma)} \le C,  \label{fus-10}
\end{align}
where $C>0$ depends only on the constants $\kappa_2$ and $\kappa_4$.
Here we can apply the embedding $\mathbb{V}^1\hookrightarrow\mathbb{X}^2$ and proceed to further apply a suitable Gr\"{o}nwall inequality to (\ref{fus-10}) so that the estimate (\ref{weak-decay}) follows with $\nu_0=c_0$ and $P_0=\frac{C}{c_0};$ indeed, after updating $c_0$ to include the embedding constant $\|U\|^2_{\mathbb{X}^2} \le C_{\overline{\Omega}}\|U\|^2_{\mathbb{V}^1}$, (\ref{fus-10}) yields, for all $t\ge 0,$
\begin{align}
\left\| U(t) \right\|^2_{\mathbb{X}^2} + \left\| \Phi^t \right\|^2_{\mathcal{M}^1_{\Omega,\Gamma}} & \le e^{-c_0 t} \left( \left\| U_0 \right\|^2_{\mathbb{X}^2} + \left\| \Phi_0 \right\|^2_{\mathcal{M}^1_{\Omega,\Gamma}} \right) + P_0  \notag \\
& \le R^2e^{-c_0 t} + P_0,  \label{fus-11}
\end{align}
where the last inequality follows from the fact that $\|\Upsilon_0\|_{\mathcal{H}^{0,1}_{\Omega,\Gamma}} \le R$.
(Also, the absolute continuity of the mapping $t\mapsto \left\| U(t) \right\|^2_{\mathbb{X}^2}+\left\| \Phi^t \right\|^2_{\mathcal{M}^1_{\Omega,\Gamma}}$ can be established as in \cite[Lemma III.1.1]{Temam88}, for example.)
The desired estimate (\ref{weak-decay}) follows from (\ref{fus-11}).
By integrating (\ref{fus-10}) over $(0,t)$ while keeping in mind the bound (\ref{fus-11}), we also find (\ref{fus-11.5}).

The existence of the bounded set $\mathcal{B}^0$ in $\mathcal{H}^{0,1}_{\Omega,\Gamma}$ that is absorbing and positively invariant for $\mathcal{S}(t)$ follows from (\ref{weak-decay}).
Indeed, define
\begin{equation*}
\mathcal{B}^0:=\left\{ (U,\Phi)\in \mathcal{H}^{0,1}_{\Omega,\Gamma} : \left\| (U,\Phi) \right\|_{\mathcal{H}^{0,1}_{\Omega,\Gamma}} \le \sqrt{P_0+1} \right\}.
\end{equation*}
Given any nonempty bounded subset $B$ in $\mathcal{H}^{0,1}_{\Omega,\Gamma}\setminus\mathcal{B}^0$, then we have that $\mathcal{S}(t)B\subseteq\mathcal{B}^0$, in $\mathcal{H}^{0,1}_{\Omega,\Gamma}$, for all $t\ge t_0$, where 
\begin{equation} \label{time0}
t_0=t_0(R):=\max\{\ln(Q(R))/c_0,0\}.
\end{equation}
This finishes the proof.
\end{proof}

\begin{corollary}
From (\ref{weak-decay}) it follows that any weak solution $(U(t),\Phi^t)$ to Problem {\textbf{P}} according to Definition \ref{d:weak-solution} is bounded uniformly in $t$. 
Indeed, for all $\Upsilon_0\in\mathcal{H}^{0,1}_{\Omega,\Gamma}$ in which $\|\Upsilon_0\|_{\mathcal{H}^{0,1}_{\Omega,\Gamma}}\le R$ for some $R>0$,
\begin{equation} \label{weak-bound}
\limsup_{t\rightarrow+\infty} \left\|\mathcal{S}(t)\Upsilon_0\right\|_{\mathcal{H}^{0,1}_{\Omega,\Gamma}} \le Q(R).
\end{equation}
\end{corollary}

We conclude this section with a local Lipschitz continuity result for the semiflow generated by the weak solutions of Problem {\textbf P}, here defined on {\em compact} time intervals and in the weaker energy space $\mathcal{H}^{-1,0}_{\Omega,\Gamma}=\mathbb{V}^{-1}\times\mathcal{M}^0_{\Omega,\Gamma}$.
(Recall this space was defined in (\ref{spaces}).)

The following result is based on \cite[Proposition 2.2]{Pata&Zelik06-2}.

\begin{theorem}  \label{t:weak-cont}
Let $T>0$ and $\Upsilon_{01}=(U_{01},\Phi_{01}),\Upsilon_{02}=(U_{02},\Phi_{02})\in\mathcal{B}^0$ (the bounded absorbing set in $\mathcal{H}^{0,1}_{\Omega,\Gamma}$).
Any two global weak solutions, $\Upsilon^1(t)$ and $\Upsilon^2(t)$, to Problem {\textbf{P}} corresponding to the initial datum $\Upsilon_{01}$ and $\Upsilon_{02}$, respectively, satisfy, for all $t\in[0,T]$, 
\begin{equation}  \label{lowc-1}
\|\Upsilon^1(t)-\Upsilon^2(t)\|_{\mathcal{H}^{-1,0}_{\Omega,\Gamma}} \le e^{(CT)t} \|\Upsilon_{01} - \Upsilon_{02}\|_{\mathcal{H}^{-1,0}_{\Omega,\Gamma}}
\end{equation}
for some constant $C>0$.
\end{theorem}

\begin{proof} 
Let $T>0.$
Let $\Upsilon_{01},\Upsilon_{02}\in\mathcal{B}^0\subset\mathcal{H}^{0,1}_{\Omega,\Gamma}$ and let $\Upsilon^1(t)=(U_1(t),\Phi^t_1)$ and $\Upsilon^2(t)=(U_2(t),\Phi^t_2)$, respectively, be the corresponding weak solutions.
Set ${\widetilde{U}}(t)=U_1(t)-U_2(t)$ and ${\widetilde{\Phi }}^t=\Phi _1^t-\Phi _2^t$. 
The function ${\widetilde{\Upsilon}}(t)=({\widetilde{U}}(t),{\widetilde{\Phi }}^t)$ satisfies the equations
\begin{align}
& \left\langle \partial_t\widetilde{U}(t),V\right\rangle _{\mathbb{X}^2}+\left\langle {\rm A_W^{0,\beta,\nu,\omega}}\widetilde{U}(t),V\right\rangle_{\mathbb{X}^2}+\left\langle
F(U_1(t))-F(U_2(t)),V\right\rangle_{\mathbb{X}^2}  \notag \\
& +\int_0^\infty\mu_\Omega(s) \left\langle {\rm A_W^{\alpha,0,0,\omega}}\widetilde{\Phi}^t(s),V\right\rangle_{\mathbb{X}^2}ds+\nu \int_0^\infty\mu_\Gamma(s)\left\langle {\rm C}\widetilde{\xi}^t(s),v_{\mid\Gamma}\right\rangle_{L^2(\Gamma)}{\rm d}s =0  \label{pp7}
\end{align}
and
\begin{equation}
\left\langle \partial_t\widetilde{\Phi}^t(s) -{\rm T_r}\widetilde{\Phi}^t(s)-\widetilde{U}(t),\Pi \right\rangle_{\mathcal{M}_{\Omega,\Gamma }^1}=0,  \label{pp8}
\end{equation}
for all $(V,\Pi) \in (\mathbb{V}^1\oplus (L^3(\Omega)\times L^{r}(\Gamma))) \times \mathcal{M}_{\Omega,\Gamma }^1$, subject to the associated initial conditions
\begin{equation*}
\widetilde{U}(0)=U_1(0) - U_2(0) \quad \text{and} \quad \widetilde{\Phi }^0=\Phi_1^0 -\Phi_2^0.
\end{equation*}

Let 
\begin{equation*}
{\widetilde{\chi}}(t):=\int_0^t {\widetilde{U}}(\tau) {\rm d}\tau \quad \text{and} \quad {\widetilde{\Psi}}^t(s):=\int_0^t {\widetilde{\Phi}}^\tau(s){\rm d}\tau \quad \text{with} \quad {\widetilde{\psi}}^t(s):=\int_0^t {\widetilde{\xi}}^\tau(s){\rm d}\tau.   \label{transformation}
\end{equation*}
Integrating equations (\ref{pp7}) and (\ref{pp8}) over $(0,t)$, we find the transformed equations
\begin{align}
& \left\langle \partial_t{\widetilde{\chi}}(t),V\right\rangle_{\mathbb{X}^2}+\left\langle {\rm A_W^{0,\beta,\nu,\omega}}{\widetilde{\chi}}(t),V\right\rangle_{\mathbb{X}^2} + \int_0^t \left\langle
F(U_1(\tau))-F(U_2(\tau)),V\right\rangle_{\mathbb{X}^2} {\rm d}\tau  \notag \\
& +\int_0^\infty \mu_\Omega(s) \left\langle {\rm A_W^{\alpha,0,0,\omega}}\widetilde{\Psi}^t(s),V\right\rangle_{\mathbb{X}^2}ds+\nu \int_0^\infty \mu_\Gamma(s)\left\langle {\rm B}\widetilde{\psi}^t(s),v_{\mid\Gamma}\right\rangle_{L^2(\Gamma)}{\rm d}s =0  \label{pp9}
\end{align}
and
\begin{equation}
\left\langle \partial_t\widetilde{\Psi}^t(s) -{\rm T_r}\widetilde{\Psi}^t(s)-\widetilde{\chi}(t),\Pi \right\rangle_{\mathcal{M}_{\Omega,\Gamma}^1}=0.  \label{pp10}
\end{equation}
Now we choose in (\ref{pp9}) and (\ref{pp10}) $V=\widetilde{\chi}(t)$ and, respectively, $\Pi=\widetilde{\Psi}^t$ to obtain the identities
\begin{align}
& \frac{1}{2}\frac{{\rm d}}{{\rm d}t}\|\widetilde{\chi}(t)\|^2_{\mathbb{X}^2}+\|\widetilde{\chi}(t)\|^2_{\mathbb{V}^1}+\int_0^t\langle F(U_1(\tau))-F(U_2(\tau)),\widetilde{\chi}(t)\rangle_{\mathbb{X}^2}{\rm d}\tau \notag \\
& + \int_0^\infty\mu_\Omega(s)\langle {{\rm A_W^{\alpha,0,0,\omega}}}{\widetilde{\Psi}}^t(s),\widetilde{\chi}(t) \rangle_{\mathbb{X}^2}ds+\nu \int_0^\infty\mu_\Gamma(s)\left\langle {\rm B}\widetilde{\psi}^t(s),\widetilde{\chi}^t_{\mid\Gamma}(s)\right\rangle_{L^2(\Gamma)}{\rm d}s =0 \label{pp9.1}
\end{align}
and
\begin{equation}
\frac{1}{2}\frac{{\rm d}}{{\rm d}t}\|\widetilde{\Psi}^t\|^2_{\mathcal{M}^1_{\Omega,\Gamma}}-\left\langle {\rm T_r}\widetilde{\Psi}^t(s),\widetilde{\Psi}^t \right\rangle_{\mathcal{M}_{\Omega,\Gamma}^1}=\left\langle\widetilde{\chi}(t),\widetilde{\Psi}^t \right\rangle_{\mathcal{M}_{\Omega,\Gamma}^1}.  \label{pp10.1}
\end{equation}
Together, (\ref{pp9.1}) and (\ref{pp10.1}) readily become the differential inequality, which holds for almost all $t\ge0$, 
\begin{align}
& \frac{{\rm d}}{{\rm d}t}\left\{ \|\widetilde{\chi}(t)\|^2_{\mathbb{X}^2} + \|\widetilde{\Psi}^t\|^2_{\mathcal{M}^1_{\Omega,\Gamma}} \right\} + 2\|\widetilde{\chi}(t)\|^2_{\mathbb{V}^1} + 2\left\langle \int_0^t F(U_1(\tau))-F(U_2(\tau)){\rm d}\tau,\widetilde{\chi}(t)\right\rangle_{\mathbb{X}^2} \le 0.     \label{pp11}
\end{align}
Since the nonlinear terms $f$ and $g$ satisfying (\ref{assm-1}) and (\ref{assm-2}) are Lipschitz on the absorbing set $\mathcal{B}^0$ (cf. e.g., \cite[Lemma 2.6]{Graber-Shomberg-16}), we have the following estimate
\begin{align}
& 2\left|\left\langle \int_0^t F(U_1(\tau))-F(U_2(\tau)){\rm d}\tau, \widetilde{\chi}(t)\right\rangle_{\mathbb{X}^2}\right|  \notag \\
& \le 2 \left|\left\langle \int_0^t f(u_1(\tau))-f(u_2(\tau)) {\rm d}\tau,\widetilde\chi(t) \right\rangle_{L^2(\Omega)}\right| {\rm d}\sigma + 2 \left|\left\langle \int_0^t g(u_1(\tau))-g(u_2(\tau)) {\rm d}\tau,\widetilde\chi(t) \right\rangle_{L^2(\Gamma)}\right|   \notag \\ 
& \le C \left( \int_0^t \|\widetilde{u}(\tau)\|_{L^2(\Omega)} {\rm d}\tau \right) \|\widetilde\chi(t)\|_{L^2(\Omega)} + C \left( \int_0^t \|\widetilde{u}(\tau)\|_{L^2(\Gamma)} {\rm d}\tau \right) \|\widetilde\chi(t)\|_{L^2(\Gamma)}  \notag \\ 
& \le C_\iota \left( \int_0^t \|\widetilde{u}(\tau)\|_{L^2(\Omega)} {\rm d}\tau \right)^2 + \iota\|\widetilde\chi(t)\|^2_{L^2(\Omega)} + C_\iota \left( \int_0^t \|\widetilde{u}(\tau)\|_{L^2(\Gamma)} {\rm d}\tau \right)^2 + \iota\|\widetilde\chi(t)\|^2_{L^2(\Gamma)} \quad (\forall\iota>0) \notag \\ 
& \le C_\iota \left( \int_0^t \|\partial_t\widetilde{\chi}(\tau)\|_{\mathbb{X}^2} {\rm d}\tau \right)^2 + \iota\|\widetilde\chi(t)\|^2_{\mathbb{X}^2}  \notag \\ 
& \le C_\iota \int_0^t \|\widetilde{\Upsilon}(\tau)\|^2_{\mathcal{H}^{0,1}_{\Omega,\Gamma}} {\rm d}\tau + \iota C \|\widetilde\chi(t)\|^2_{\mathbb{V}^1}.  \notag \\
& \le C_\iota T \|\widetilde{\Upsilon}(t)\|^2_{\mathcal{H}^{0,1}_{\Omega,\Gamma}} + \iota C \|\widetilde\chi(t)\|^2_{\mathbb{V}^1}.  \label{wkea16}
\end{align}
The constant $C_\iota>0$ satisfies $C_\iota\sim\iota^{-1}$.
With (\ref{wkea16}), the inequality (\ref{pp11}) becomes, for a suitably small $\iota>0,$
\begin{align}
& \frac{{\rm d}}{{\rm d}t}\|\widetilde{\Upsilon}(t)\|^2_{\mathcal{H}^{0,1}_{\Omega,\Gamma}} + \|\widetilde{\chi}(t)\|^2_{\mathbb{V}^1} \le CT \|\widetilde{\Upsilon}(t)\|^2_{\mathcal{H}^{0,1}_{\Omega,\Gamma}} .     \label{pp12}
\end{align}
Now integrating (\ref{pp12}) over $(0,t)$ yields, for all $t\in(0,T),$
\begin{align}
\|\widetilde{\Upsilon}(t)\|^2_{\mathcal{H}^{0,1}_{\Omega,\Gamma}} + \int_0^t\|\widetilde{\chi}(\tau)\|^2_{\mathbb{V}^1}{\rm d}\tau & \le \|\widetilde{\Upsilon}(0)\|^2_{\mathcal{H}^{0,1}_{\Omega,\Gamma}}  + CT \int_0^t \|\widetilde{\Upsilon}(\tau)\|^2_{\mathcal{H}^{0,1}_{\Omega,\Gamma}} {\rm d}\tau.  \label{wkea155}
\end{align}
Observe, 
\begin{align}
\|\widetilde\Upsilon(0)\|^2_{\mathcal{H}^{0,1}_{\Omega,\Gamma}} & = \|(\widetilde{U}(0),\widetilde{\Psi}^0)\|^2_{\mathcal{H}^{0,1}_{\Omega,\Gamma}}  \notag \\
& = \|\widetilde{U}_0\|^2_{\mathbb{X}^2}+\|\widetilde{\Psi}_0\|^2_{\mathcal{M}^1_{\Omega,\Gamma}}  \notag \\
& \le \|\widetilde{\Upsilon}_0\|^2_{\mathcal{H}^{-1,0}_{\Omega,\Gamma}}.  \label{wkea156}
\end{align}
Thus, (\ref{wkea155})--(\ref{wkea156}) become
\begin{align}
\|\widetilde{\Upsilon}(t)\|^2_{\mathcal{H}^{0,1}_{\Omega,\Gamma}} + \int_0^t\|\widetilde{\chi}(\tau)\|^2_{\mathbb{V}^1}{\rm d}\tau & \le \|\widetilde{\Upsilon}_0\|^2_{\mathcal{H}^{-1,0}_{\Omega,\Gamma}} + CT \int_0^t \|\widetilde{\Upsilon}(\tau)\|^2_{\mathcal{H}^{0,1}_{\Omega,\Gamma}} {\rm d}\tau.  \label{wkea17}
\end{align}
Finally, omitting the positive integral on the left-hand side of (\ref{wkea17}) and applying the integral form of Gr\"{o}nwall's lemma to the result produces the claim (\ref{lowc-1}).
\end{proof}

\section{Weak exponential attractors} \label{s:wkea}

This section is motivated by \cite[\S4]{Pata&Zelik06-2}.
In this section we show the existence of a so-called {\em{weak}} exponential attractor. 
We seek a weak exponential attractor in place of the standard one because of the issue raised by Remark \ref{r:trace-map}. 
In order to obtain the desired compactness from the attractor, we need to rely on the compactness of the embedding $\mathcal{K}^1_{\Omega,\Gamma}\hookrightarrow\mathcal{M}^0_{\Omega,\Gamma}.$
Hence, we rely on the phase space $\widehat{\mathcal{H}}^{0,1}_{\Omega,\Gamma} = \mathbb{X}^2\times\mathcal{K}^1_{\Omega,\Gamma}$ and its compact injection into the weak topology of $\mathcal{H}^{0,1}_{\Omega,\Gamma}$, that is the space $\mathcal{H}^{-1,0}_{\Omega,\Gamma}=\mathbb{V}^2\times(\mathcal{M}^2_{\Omega,\Gamma}\cap {\rm D(T_r)})$ described above in (\ref{spaces}).
In order to provide such precompact trajectories, we need to bound the memory terms in the more regular space $\mathcal{K}^1_{\Omega,\Gamma}$, and this requires us to satisfy the shared hypothesis of Lemmata \ref{what-1} and \ref{what-2} above; that is, $\|U(t)\|_{\mathbb{V}^1}\le K$, for some constant $K>0$ for all $t\ge0$. 
This in turn requires the use of the so-called quasi-strong solutions (see Definition \ref{d:strong-solution}).
Hence, trajectories with data in 
\[
\mathbb{H}^2:=\mathbb{V}^2\times \left(\mathcal{M}^2_{\Omega,\Gamma}\cap{\rm D}(\rm T_r)\right)
\] 
are sufficiently smooth quasi-strong solutions which are bounded in the norm of $\widehat{\mathcal{H}}^{0,1}_{\Omega,\Gamma}$ and precompact in $\mathcal{H}^{-1,0}_{\Omega,\Gamma}$ (recall (\ref{spaces})).
Recently, works such as \cite{CGGM10,Gal07,Gal&Grasselli08,Gal-Shomberg15,GrMS10,Miranville&Zelik05,ShombergXX} are able to establish the existence of an exponential attractor for wave equations with dynamic boundary conditions with the use of suitable $H^2$-elliptic regularity estimates. 
Without bona-fide strong solutions to Problem {\bf P}, similar estimates are not available here.

\begin{theorem}  \label{t:wk-exp-attr}
The semigroup of solution operators $\mathcal{S}=(\mathcal{S}(t))_{t\ge0}$ generated by the quasi-strong solutions of Problem \textbf{P} admits a weak exponential attractor $\mathcal{E}^{-1}$ that satisfies:
\begin{enumerate}
\item $\mathcal{E}^{-1}$ is bounded in $\widehat{\mathcal{H}}^{0,1}_{\Omega,\Gamma}$ and compact in $\mathcal{H}^{-1,0}_{\Omega,\Gamma}$,
\item $\mathcal{E}^{-1}$ is positively invariant; i.e., for all $t\ge0,$ $S(t)\mathcal{E}^{-1}\subseteq\mathcal{E}^{-1},$
\item $\mathcal{E}^{-1}$ attracts bounded subsets of $\widehat{\mathcal{H}}^{0,1}_{\Omega,\Gamma}$ exponentially with the metric of $\mathcal{H}^{-1,0}_{\Omega,\Gamma}$; i.e., there exists $\rho>0$ and $Q$ such that, for every bounded subset $B\subset\widehat{\mathcal{H}}^{0,1}_{\Omega,\Gamma}$ and for all $t\ge0,$
\[
{\rm dist}_{\mathcal{H}^{-1,0}_{\Omega,\Gamma}}(\mathcal{S}(t)B,\mathcal{E}^{-1})\le Q(\|B\|_{\widehat{\mathcal{H}}^{0,1}_{\Omega,\Gamma}})e^{-\rho t}.
\]
\item $\mathcal{E}^{-1}$ possesses finite fractal dimension in $\mathcal{H}^{-1,0}_{\Omega,\Gamma}$; i.e.,
\begin{equation*}  \label{fdimension}
{\rm dim_F}(\mathcal{E}^{-1},\mathcal{H}^{-1,0}_{\Omega,\Gamma}):=\limsup_{r\rightarrow 0}\frac{\ln \mu_{\mathcal{H}^{-1,0}_{\Omega,\Gamma}}(\mathcal{E}^{-1},r)}{-\ln r}<\infty,
\end{equation*}
where $\mu _{\mathcal{H}^{-1,0}_{\Omega,\Gamma}}(\mathcal{E}^{-1},r)$ denotes the minimum number of balls of radius $r$ from $\mathcal{H}^{-1,0}_{\Omega,\Gamma}$ required to cover $\mathcal{E}^{-1}$.
\end{enumerate}
\end{theorem}

The proof of Theorem \ref{t:wk-exp-attr} follows from the application of an abstract result modified only to suit our needs here (for further reference, see for example, \cite{EFNT95,EMZ00,GGMP05}).

\begin{proposition}  \label{abstract1}
Let $H_0$ and $H_{-1}$ be Hilbert spaces such that the embedding $H_0\hookrightarrow H_{-1}$ is compact. 
Let $S=(S(t))_{t\ge0} $ be a semigroup of operators on $H_0$. 
Assume the following hypotheses hold:

\begin{enumerate}
\item[(H1)] There exists a bounded absorbing set $B_0\subset H_0$ which is positively invariant for $S(t).$ 
More precisely, there exists a time $t_0>0$ (possibly depending on the radius of $B_0$) such that, for all $t\ge t_0,$
\begin{equation*}
S(t)B_0\subset B_0.
\end{equation*}

\item[(H2)] There is $t^*\ge t_0$ such that the map $S(t^*)$ admits the decomposition, for all $\zeta_{01},\zeta _{02}\in B_0,$ 
\begin{equation*}
S(t^*)\zeta_{01} - S(t^*)\zeta_{02} = L(\zeta_{01},\zeta_{02}) + K(\zeta_{01},\zeta_{02}),
\end{equation*}
where, for some constants $\kappa=\kappa(t^*) \in (0,\frac{1}{2})$ and $\Lambda=\Lambda(t^*)\ge 0$, the following hold:
\begin{equation}  \label{H2-L}
\|L(\zeta_{01},\zeta_{02})\|_{H_{-1}} \le \kappa \|\zeta_{01}-\zeta_{02} \|_{H_{-1}}
\end{equation}
and
\begin{equation}  \label{H2-K}
\|K(\zeta_{01},\zeta_{02})\|_{H_0} \le \Lambda \|\zeta_{01}-\zeta_{02} \|_{H_{-1}}.
\end{equation}

\item[(H3)] The map
\begin{equation*}  \label{to-H3}
(t,\zeta_0)\mapsto S(t)\zeta:[t^*,2t^*]\times B_0\rightarrow B_0
\end{equation*}
is Lipschitz continuous on $B_0$ in the topology of $H_{-1}$.
\end{enumerate}

Then the semigroup $S$ admits an exponential attractor $E_{-1}$ in $B_0.$
\end{proposition}

To begin, note that the embedding $\widehat{\mathcal{H}}^{0,1}_{\Omega,\Gamma} \hookrightarrow \mathcal{H}^{-1,0}_{\Omega,\Gamma}$ is compact (see (\ref{spaces}) above) due to the fact that the embedding $\mathbb{V}^s\hookrightarrow\mathbb{X}^2$ is compact, for any $s>0$, and by the introduction of the ``tail spaces'' above, $\mathcal{K}^{i}_{\Omega,\Gamma}$, $i\in\{0,1\}$, where the embedding $\mathcal{K}^1_{\Omega,\Gamma} \hookrightarrow \mathcal{M}^{0}_{\Omega,\Gamma}$ is compact. 
As already mentioned above, we will require Lemmata \ref{what-1} and \ref{what-2} in order to provide the precompactness of the trajectories of Problem {\bf P}.
This in turn means we need to satisfy the hypothesis that $\|U(t)\|_{\mathbb{V}^1}\le K$, for some constant $K>0$ for all $t\ge0.$
Hence, the exponential attractors we seek are for the quasi-strong solutions.

Next we show (H1) of Proposition \ref{abstract1} holds.

\begin{remark}
We claim the space $\mathbb{V}^2\times (\mathcal{M}^2_{\Omega,\Gamma} \cap {\rm D(T_r)})$ is {\em dense} in the space $\mathcal{H}^{0,1}_{\Omega,\Gamma}=\mathbb{X}^2\times\mathcal{M}^1_{\Omega,\Gamma}$; i.e., where the bounded absorbing set $\mathcal{B}^0$ resides.
This is evident from the fact that $\mathbb{V}^2$ is the domain of the (densely defined) operator ${\rm A_W^{\alpha,\beta,\nu,\omega}}$ and also from the fact that the space also contains the domain of a densely defined operator, ${\rm T_r}$.
This means every element in $\mathcal{B}^0\subset\mathcal{H}^{0,1}_{\Omega,\Gamma}$ is the limit a sequence $\{\Upsilon_0^{(n)}\}_{n=1}^\infty \subset \mathbb{V}^2\times (\mathcal{M}^2_{\Omega,\Gamma} \cap {\rm D(T_r)}).$
So in the sequel, when we establish conditions (H1) and (H2) in Lemma \ref{t:quasi-ball} and Lemma \ref{t:to-H2} below, it suffices to take initial data from $\mathbb{V}^2\times (\mathcal{M}^2_{\Omega,\Gamma} \cap {\rm D(T_r)})$ and work with the corresponding quasi-strong solutions.
\end{remark}

\begin{lemma}  \label{t:quasi-ball}
Suppose the assumptions of Theorem \ref{t:quasi-strong} and Lemma \ref{t:weak-ball} hold. 
For all $R>0$ and $\Upsilon_0=(U_0,\Phi_0)\in\mathbb{H}^2=\mathbb{V}^2\times (\mathcal{M}^2_{\Omega,\Gamma} \cap {\rm D(T_r)})$ with $\|\Upsilon_0\|_{\widehat{\mathcal{H}}^{0,1}_{\Omega,\Gamma}}\le R$, there exist positive constants $\hat c_0=\min\{c_0,\delta_\Omega,\delta_\Gamma\}$ and $P_0$, and a positive monotonically increasing function $Q(\cdot)$, such that, for all $t\ge0$,
\begin{align}
    \left\| \left( U(t),\Phi^t \right) \right\|^2_{\widehat{\mathcal{H}}^{0,1}_{\Omega,\Gamma}} \le Q(R)e^{-\hat c_0 t}(t+1)+P_0  \label{strong-decay} 
\end{align}
(where $c_0$ and $P_0$ are due to Lemma \ref{t:weak-ball}).
In addition, the bounded absorbing set $\mathcal{B}^0$ in $\mathcal{H}^{0,1}_{\Omega,\Gamma}$ (given in Theorem \ref{t:weak-ball}) is bounded and absorbing in $\widehat{\mathcal{H}}^{0,1}_{\Omega,\Gamma}=\mathbb{X}^2\times\mathcal{K}^1_{\Omega,\Gamma}$.
Moreover, (H1) holds.
\end{lemma}

\begin{proof}
Let $R>0$ and $\Upsilon_0=(U_0,\Phi_0)\in\mathbb{H}^2=\mathbb{V}^2\times (\mathcal{M}^2_{\Omega,\Gamma} \cap {\rm D(T_r)})$ be such that $\|\Upsilon_0\|_{\widehat{\mathcal{H}}^{0,1}_{\Omega,\Gamma}}\le R.$
Recall the definition of the norm
\begin{equation} \label{fus-12}
\|\Phi\|^2_{\mathcal{K}^1_{\Omega,\Gamma}} = \|\Phi\|^2_{\mathcal{M}^1_{\Omega,\Gamma}} + \|\partial_s\Phi\|^2_{\mathcal{M}^1_{\Omega,\Gamma}} + \sup_{\tau\ge1} \tau\mathbb{T}(\tau;\Phi). 
\end{equation}
Hence, in light of Lemmata \ref{what-1} and \ref{what-2}, it suffices to show there exists a constant $K>0$ such that, for all $t\ge0,$ 
\begin{equation} \label{desire}
\|U(t)\|_{\mathbb{V}^1}\le K.
\end{equation}
To do this we will report from the proof of \cite[Theorem 3.13]{Gal-Shomberg15-2}.

We consider the time-differentiated Problem {\textbf P} where we now seek a function $(U,\Phi)$ satisfying 
\begin{align}
& \left\langle \partial_{tt}U,\Xi \right\rangle _{\mathbb{X}^2} + \left\langle \mathrm{A_W^{0,\beta ,\nu,\omega}}\partial_tU,\Xi \right\rangle_{\mathbb{X}^2} + \left\langle \partial_t\Phi^t,\Xi \right\rangle_{\mathcal{M}^1_{\Omega,\Gamma}}  \label{ap1} \\
& = - \left\langle f'(u) \partial_tu,\varsigma \right\rangle_{L^2(\Omega)} - \left\langle \widetilde{g}'(u) \partial_tu,\varsigma_{\mid\Gamma} \right\rangle_{L^2(\Gamma)}  \notag
\end{align}
and
\begin{equation}
\left\langle \partial_{tt}\Phi^t,\Pi\right\rangle_{\mathcal{M}^1_{\Omega,\Gamma}}=\left\langle \mathrm{T_r}\partial_t\Phi^t,\Pi\right\rangle_{\mathcal{M}^1_{\Omega,\Gamma}} +\left\langle \partial_tU,\Pi\right\rangle_{\mathcal{M}^1_{\Omega,\Gamma}}  \label{ap2}
\end{equation}
hold for almost all $t\in (0,T)$, for any $T>0$, and for all $\Xi=(\varsigma,\varsigma_{\mid\Gamma})^{\mathrm{tr}}\in\mathbb{V}^1\cap\left( L^2(\Omega)\times L^r(\Gamma) \right)$ and $\Pi=(\rho,\rho_{\mid\Gamma})^{\mathrm{tr}}\in \mathcal{M}^1_{\Omega,\Gamma}$.
Moreover, the function $(U,\Phi) $ fulfills the conditions $U(0)=U_0,$ $\Phi^0=\Phi_0$ and 
\begin{equation}
\partial_tU(0)=\widehat{U}_0, \quad\quad \partial_t\Phi^0=\widehat{\Phi}^0,  \label{ap3}
\end{equation}
where we have set
\begin{align*}
\widehat{U}_0 & := -\mathrm{A_W^{0,\beta,\nu,\omega}}U_0 - \int_0^\infty\mu_\Omega(s)\mathrm{A_W^{\alpha,0,0,\omega}}\Phi_0(s){\rm d}s - \nu \int_0^\infty\mu_\Gamma(s)\binom{0}{\mathrm{B}\xi_0(s)}{\rm d}s - F(U_0),   \\
\widehat{\Phi}^0 & :=\mathrm{T_r}\Phi_0(s) + U_0.
\end{align*}
Note that since $U_0\in \mathbb{V}^2$ and $\Phi^0\in \mathcal{M}^2_{\Omega,\Gamma} \cap \mathrm{D(T_r)}$, then $(\widehat{U}_0,\widehat{\Phi}^0)\in \mathbb{X}^2\times \mathcal{M}^1_{\Omega,\Gamma}=\mathcal{H}^{0,1}_{\Omega,\Gamma}$, owing to the continuous embeddings $H^2(\Omega) \hookrightarrow L^\infty(\Omega)$ and $H^2(\Gamma) \hookrightarrow L^\infty(\Gamma)$. 
According to assumptions (\ref{mu-1})-(\ref{mu-3}), we can infer that
\begin{equation*} \label{qest15}
0\le \int_0^\infty\mu_S(s)ds=:m_S<\infty, \quad \text{for each}\ S\in \left\{ \Omega,\Gamma \right\}, 
\end{equation*}
such that repeated application of Jensen's inequality yields
\begin{align*}
\left\| \int_0^\infty\mu_\Omega(s)\mathrm{A_W^{\alpha,0,0,\omega}}\Phi_0(s)ds\right\|_{\mathbb{X}^2}^2 & \le m_\Omega \int_0^\infty \mu_\Omega(s)\left\| \mathrm{A_W^{\alpha,0,0,\omega}}\Phi_0(s)\right\|_{\mathbb{X}^2}^2{\rm d}s \\
& \le Cm_\Omega \int_0^\infty \mu_\Omega(s)\|\Phi_0(s)\|^2_{H^2(\Omega)}ds
\end{align*}
and
\begin{align*}
\left\| \int_0^\infty\mu_\Gamma(s)\mathrm{B}\xi_0(s){\mathrm{d}}s\right\|^2_{L^2(\Gamma)} & \le m_\Gamma \int_0^\infty\mu_\Gamma(s)\left\| \mathrm{B}\xi_0(s)\right\|^2_{L^2(\Gamma)}{\rm d}s \\
& \le Cm_\Gamma\int_0^\infty\mu_\Gamma(s)\left\| \Phi_0(s)\right\|^2_{H^2(\Gamma)}{\rm d}s.
\end{align*}
We proceed to take $\Xi=\partial_tU(t)$ in (\ref{ap1}) and $\Pi=\partial_t\Phi^t(s) $ in (\ref{ap2}).
By the definition of quasi-strong solution, such a choice of test function is admissible. 
Summing the resulting identities and using (\ref{quasi-assm-3})-(\ref{quasi-assm-4}), we obtain
\begin{align}
\frac{1}{2}\frac{{\rm d}}{{\rm d}t} & \left\{ \|\partial_tU\|_{\mathbb{X}^2}^2 + \|\partial_t\Phi^t\|_{\mathcal{M}^1_{\Omega,\Gamma}}^2 \right\} - \left\langle \mathrm{T_r}\partial_t\Phi^t,\partial_t\Phi^t \right\rangle_{\mathcal{M}^1_{\Omega,\Gamma}}  \notag \\
& + \omega\|\nabla \partial_tu\|^2_{L^2(\Omega)} + \nu\|\nabla_{\Gamma}\partial_tu\|^2_{L^2(\Gamma)} + \beta\|\partial_tu\|^2_{L^2(\Gamma)}  \notag \\
& = -\left\langle f'(u) \partial_tu,\partial_tu\right\rangle_{L^2(\Omega)} - \left\langle \widetilde{g}'(u) \partial_tu,\partial_tu\right\rangle_{L^2(\Gamma)}
\notag \\
& \le \max\{M_f,M_g\} \|\partial_tU\|^2_{\mathbb{X}^2}.  \label{qest16}
\end{align}
Thus, integrating (\ref{qest16}) over $(0,t),$ then by application of Gr\"owall's inequality we arrive at the estimate
\begin{equation} \label{qest17}
\|\partial_tU\|_{\mathbb{X}^2}^2 + \|\partial_t\Phi^t\|_{\mathcal{M}^1_{\Omega,\Gamma}}^2 + \int_0^t \left( 2\|\partial_tU(\tau)\|^2_{\mathbb{V}^1} + \|\partial_t\Phi^\tau\|^2_{L_{k_\Omega\oplus k_\Gamma}(\mathbb{R}_+;\mathbb{V}^1)} \right) {\rm d}\tau \le Q(R),
\end{equation}
for all $t\ge 0$ and all $R>0$ such that $\|(U_0,\Phi_0)\|_{\mathcal{H}^{2,2}_{\Omega,\Gamma}} \le R$.

We now establish a bound for $U$ in $L^\infty(0,\infty;\mathbb{V}^1)$; that is, (\ref{desire}). 
To this end, we proceed to take $\Xi=U(t)$ in (\ref{ap1}) in order to derive
\begin{align}
& \frac{{\rm d}}{{\rm d}t}\left\{ \|U\|^2_{\mathbb{V}^1} + \left\langle \partial_tU,U\right\rangle_{\mathbb{X}^2} + 2\int_\Omega h_f(u) {\rm d}x + 2\int_\Gamma h_g(u) {\rm d}\sigma \right\}  \notag \\
& =2\|\partial_tU\|^2_{\mathbb{X}^2} - 2\left\langle \partial_t\Phi^t,U\right\rangle_{\mathcal{M}^1_{\Omega,\Gamma}}.  \label{qest18}
\end{align}
Moreover, using the Cauchy-Schwarz and Young inequalities, and (\ref{quasi-assm-7})-(\ref{quasi-assm-8}), the following basic inequality holds:
\begin{align}
& C_\ast\|U\|^2_{\mathbb{V}^1}-Q(R)  \notag \\
& \le \|U\|^2_{\mathbb{V}^1}+\left\langle \partial_tU,U\right\rangle_{\mathbb{X}^2}+2\int_\Omega h_f(u) {\rm d}x+2\int_\Gamma h_g(u) {\rm d}\sigma \notag \\
& \le C\|U\|^2_{\mathbb{V}^1}+Q(R),  \label{qest19tris}
\end{align}
for some constants $C_\ast, C>0$ and some function $Q(\cdot)>0,$ all independent of $t$. 
Finally, for any $\iota >0$ we estimate
\begin{align}
& -\left\langle \partial_t\Phi^t,U\right\rangle_{\mathcal{M}^1_{\Omega,\Gamma}} \notag \\
& \le \iota \|U\|^2_{\mathbb{V}^1} + C_\iota\int_0^\infty\mu_\Omega(s)\|\partial_t\eta(s)\|^2_{H^1(\Omega)}{\rm d}s + C_\iota\int_0^\infty\mu_\Gamma(s)\|\partial_t\xi(s)\|^2_{H^1(\Gamma)}{\rm d}s  \notag \\
& \le \iota \|U\|^2_{\mathbb{V}^1} - C_\iota \delta_\Omega^{-1}\int_0^\infty\mu'_\Omega(s)\|\partial_t\eta(s)\|^2_{H^1(\Omega)}{\rm d}s - C_\iota\delta_\Gamma^{-1}\int_0^\infty\mu'_\Gamma(s)\|\partial_t\xi(s)\|^2_{H^1(\Gamma)}{\rm d}s,  \label{qest20}
\end{align}
where in the last line we have employed assumption (\ref{mu-4}). 
Thus, from (\ref{qest18}) we obtain the inequality, for almost all $t\ge0,$ 
\begin{align}
& \frac{{\rm d}}{{\rm d}t}\left\{ \|U\|^2_{\mathbb{V}^1} + \left\langle \partial_tU,U\right\rangle_{\mathbb{X}^2} + 2\int_\Omega h_f(u) {\rm d}x + 2\int_\Gamma h_g(u) {\rm d}\sigma \right\}  \notag \\
& \le C_\iota\|U(t)\|_{\mathbb{V}^1}^2 + 2\|\partial_tU\|^2_{\mathbb{X}^2} - 2\left\langle \partial_t\Phi^t,U\right\rangle_{\mathcal{M}^1_{\Omega,\Gamma}}.  \label{qest21}
\end{align}
We now observe that $2\|\partial_tU\|^2_{\mathbb{X}^2} - 2\langle \partial_t\Phi^t,U\rangle_{\mathcal{M}^1_{\Omega,\Gamma}} \in L^1(0,\infty)$ on account of (\ref{qest17}) and (\ref{qest19tris})-(\ref{qest20}), because $\partial_tU(0) \in\mathbb{X}
^2$ by (\ref{ap3}). 
Thus, observing (\ref{qest19tris}), the application of Gronwall's inequality to (\ref{qest21}) yields the desired uniform bound (\ref{desire}).

To finish the estimate (\ref{weak-decay}) in the space $\widehat{\mathcal{H}}^{0,1}_{\Omega,\Gamma}$, we need to treat the norm (\ref{fus-12}).
Thanks to the bound (\ref{desire}) we now combine the results of Lemmata \ref{what-1} and \ref{what-2} to which yields, for the remaining two terms of (\ref{fus-12}), for all $t\ge0,$
\begin{align}
& \|{\rm T_r}\Phi^t\|^2_{\mathcal{M}^1_{\Omega,\Gamma}} + \sup_{\tau\ge1} \tau\mathbb{T}(\tau;\Phi^t) \notag \\
& \le e^{-\min\{\delta_\Omega,\delta_\Gamma\} t}\|{\rm T_r}\Phi_0\|^2_{\mathcal{M}^1_{\Omega,\Gamma}} + Q(R) \left( \|\mu_\Omega\|_{L^1(\mathbb{R}_+)} + \|\mu_\Gamma\|_{L^1(\mathbb{R}_+)} \right) \notag \\
& + 2 \left( t+2 \right)e^{-\min\{\delta_\Omega,\delta_\Gamma\} t} \sup_{\tau\ge1} \tau\mathbb{T}(\tau;\Phi_0) + Q(R)  \notag \\
& \le e^{-\delta t} \left( \left\|{\rm T_r}\Phi_0\right\|^2_{\mathcal{M}^1_{\Omega,\Gamma}} + 2(t+2)\sup_{\tau\ge1}\tau\mathbb{T}(\tau;\Phi_0)\right) + Q(R)\left( \|\mu_\Omega\|_{L^1(\mathbb{R}_+)}+\|\mu_\Gamma\|_{L^1(\mathbb{R}_+)} + 1\right)  \notag \\
& \le Q(R) \left( e^{-\min\{\delta_\Omega,\delta_\Gamma\} t}(t+1) + 1 \right).  \label{fus-13}
\end{align}
Adding (\ref{fus-13}) into (\ref{fus-11}) produces, for all $t\ge 0,$
\begin{align}
\left\| U(t) \right\|^2_{\mathbb{X}^2} + \left\| \Phi^t \right\|^2_{\mathcal{K}^1_{\Omega,\Gamma}} & \le R^2e^{-c_0 t} + P_0 + Q(R) \left( e^{-\min\{\delta_\Omega,\delta_\Gamma\} t}(t+1) + 1 \right) \notag \\
& \le Q(R)e^{-\min\{c_0,\delta_\Omega,\delta_\Gamma\}t}(t+1)+P_0,
\end{align} 
that is, (\ref{strong-decay}) holds.
\end{proof}

Throughout the remainder of the article, we denote by $\widehat{\mathcal{B}}^0$ the bounded absorbing set in $\widehat{\mathcal{H}}^{0,1}_{\Omega,\Gamma}$.

We now show that the hypotheses (H2) and (H3) hold for the semiflow $\mathcal{S}$ generated by the quasi-strong solutions of Problem \textbf{P} for the space $\widehat{\mathcal{H}}^{0,1}_{\Omega,\Gamma}$. 
Moving forward, we now show (H2) by making the appropriate ``lower-order'' estimates in the norm of $\mathcal{H}^{-1,0}_{\Omega,\Gamma}.$ 

\begin{lemma} \label{t:to-H2} 
Suppose the assumptions of Theorem \ref{t:quasi-strong} and Lemma \ref{t:weak-ball} hold.
In addition, assume there holds
\begin{equation} \label{assk2}
k_\Gamma(0)<\frac{2}{1-\nu}.
\end{equation}
Then condition (H2) holds. 
\end{lemma}

\begin{proof}
Let $\Upsilon_{01}=(U_{01},\Phi_{01}),\Upsilon_{02}=(U_{02},\Phi_{02}) \in\mathbb{V}^2\times (\mathcal{M}^2_{\Omega,\Gamma} \cap {\rm D(T_r)})$ be such that $\Upsilon_{01},\Upsilon_{02}\in\widehat{\mathcal{B}}^0$. 
For $t>0$, let $\Upsilon^1(t)=(U_1(t),\Phi^t_1)$ and $\Upsilon^2(t)=(U_2(t),\Phi^t_2)$ denote the corresponding global solutions of Problem {\textbf{P}} with the initial datum $\Upsilon_{01}$ and $\Upsilon_{02}$, respectively.
For all $t>0$, set
\begin{align}
\bar\Upsilon(t) & := \Upsilon^1(t) - \Upsilon^2(t)  \notag \\
& = \left( U_1(t),\Phi^t_1 \right) - \left( U_2(t),\Phi^t_2 \right)  \notag \\
& =: \left( \bar U(t),\bar \Phi^t \right),  \notag
\end{align}
and 
\begin{align}
\bar\Upsilon_0 & := \Upsilon_{01} - \Upsilon_{02}  \notag \\ 
& = \left( U_{01},\Phi_{01} \right) - \left( U_{02},\Phi_{02} \right) \notag \\
& = \left( U_{01}-U_{02},\Phi_{01}-\Phi_{02} \right)  \notag \\
& =: \left( \bar U_0,\bar \Phi_0 \right).  \notag
\end{align}
For each $t\ge 0$, decompose the difference $\bar{\Upsilon}(t):=\Upsilon^1(t)-\Upsilon^2(t)$ with $\bar{\Upsilon}_0:=\Upsilon_{01}-\Upsilon_{02}$ as follows,
\begin{align}
\bar{\Upsilon}(t) & = \left( \bar U(t),\bar \Phi^t \right)  \notag \\ 
& = \left( \bar V(t),\bar \Psi^t \right) + \left( \bar W(t),\bar \Theta^t \right),  \notag \\ 
& =: \bar{\Lambda}(t)+\bar{\Xi}(t),  \notag 
\end{align}
where $\bar{\Lambda}(t)$ and $\bar{\Xi}(t)$ are solutions of the following problems in $\bar V(t)=(\bar v_\Omega(t),\bar v_\Gamma(t))$, $\bar\Psi^t=(\bar\psi_\Omega^t,\bar\psi_\Gamma^t)$, $\bar W(t)=(\bar w_\Omega(t),\bar w_\Gamma(t))$ and $\bar\Theta^t=(\bar\vartheta_\Omega^t,\bar\vartheta_\Gamma^t)$: 
\begin{align}
\partial_t \bar V(t) & + {{\rm A_W^{0,\beta,\nu,\omega}}} \bar V(t) + \int_0^\infty \mu_\Omega(s) {{\rm A_W^{\alpha,0,0,\omega}}}\bar\Psi^t(s) {\rm d}s + \nu\int_0^\infty \mu_\Gamma(s) {\rm B}\psi_\Gamma^t(s) {\rm d}s =0,  \label{vsol1}
\end{align}
\begin{equation}  \label{vsol2}
\partial_t\bar\Psi^t={\rm T_r}\bar\Psi^t+\bar V(t),  
\end{equation}
\begin{equation*} \label{vsol3}
\bar\Lambda(0)=\bar\Upsilon_0,
\end{equation*}
and
\begin{align}
\partial_t \bar W(t) & + {{\rm A_W^{0,\beta,\nu,\omega}}} \bar W(t) + \int_0^\infty \mu_\Omega(s) {{\rm A_W^{\alpha,0,0,\omega}}}\bar\Theta^t(s) {\rm d}s + \nu\int_0^\infty \mu_\Gamma(s) {\rm B}\theta_\Gamma^t(s) {\rm d}s \notag \\
& + F(U_1(t))-F(U_2(t))=0,  \label{wsol1}
\end{align}
\begin{equation}  \label{wsol2}
\partial_t\bar\Theta^t={\rm T_r}\bar\Theta^t+\bar W(t),  
\end{equation}
\begin{equation*}
\bar\Xi(0)={\bf 0}. \label{wsol3}
\end{equation*}

{\bf Step 1.} Proof of (\ref{H2-L}).
Let $\varepsilon\in(0,1)$ to be chosen later and define
\begin{equation}  \label{tranv}
\bar V^*(t):=\int_0^t e^{-\varepsilon(t-\tau)}\bar V(\tau){\rm d}\tau \quad \text{and} \quad \bar\Psi^{*t}(s):=\int_0^t e^{-\varepsilon(t-\tau)}\bar\Psi^\tau(s){\rm d}\tau. 
\end{equation}
Observe, $\partial_t\bar V^* +\varepsilon\bar V^*=\bar V^*$ with $\bar V^*(0)=0$, and $\partial_t\bar \Psi^{*t} +\varepsilon\bar \Psi^{*t}=\bar \Psi^{*t}$ with $\bar \Psi^{*t}(0)=0$.
Multiply equations (\ref{vsol1})-(\ref{vsol2}) by $e^{-\varepsilon(t-\tau)}$ and integrate with respect to $\tau$ over $(0,t)$ to find,
\begin{align}
\partial_t \bar V^*(t) & + {{\rm A_W^{0,\beta,\nu,\omega}}} \bar V^*(t) + \int_0^\infty \mu_\Omega(s) {{\rm A_W^{\alpha,0,0,\omega}}}\bar\Psi^{*t}(s) {\rm d}s + \nu\int_0^\infty \mu_\Gamma(s) {\rm B}\bar\psi_\Gamma^{*t}(s) {\rm d}s =0,  \label{dc1}
\end{align}
\begin{equation}  \label{dc2}
\partial_t\bar\Psi^{*t}={\rm T_r}\bar\Psi^{*t}+\bar V^*(t),  
\end{equation}
Multiplying (\ref{dc1}) by $\bar V^*$ in $\mathbb{X}^2$ and (\ref{dc2}) by ${\rm A_W^{0,\beta,\nu,\omega}}\bar\Psi^{*t}$ in $\mathcal{M}^0_{\Omega,\Gamma}=L^2_{\mu_\Omega\oplus\mu_\Gamma}(\mathbb{R}_+;\mathbb{X}^2)$, we easily obtain the differential identities,
\begin{align}
& \frac{{\rm d}}{{\rm d}t}\|\bar V^*\|^2_{\mathbb{X}^2} + 2\|\bar V^*\|^2_{\mathbb{V}^1}  + 2\int_0^\infty \mu_\Omega(s) \langle {{\rm A_W^{\alpha,0,0,\omega}}} \bar\Psi^{*t}(s), \bar V^*(t) \rangle_{\mathbb{X}^2} {\rm d}s \notag \\
& + 2\nu\int_0^\infty \mu_\Gamma(s) \langle {\rm B}\bar\psi_\Gamma^{*t}(s), \bar v^*_\Gamma(t) \rangle_{L^2(\Gamma)} {\rm d}s =0  \label{dc3}
\end{align}
and
\begin{align}
\frac{{\rm d}}{{\rm d}t}\|\bar\Psi^{*t}\|^2_{\mathcal{M}^1_{\Omega,\Gamma}} = 2\langle {\rm T_r}\bar\Psi^{*t},\bar\Psi^{*t} \rangle_{\mathcal{M}^1_{\Omega,\Gamma}} + 2\langle \bar V^*(t),\bar\Psi^{*t} \rangle_{\mathcal{M}^1_{\Omega,\Gamma}}.   \label{dc4}
\end{align}
Of course we employ the basic estimate from (\ref{operator-T-1})
\begin{equation} \label{dc5}
-\langle {\rm T_r}\bar\Psi^{*t},\bar\Psi^{*t} \rangle_{\mathcal{M}^1_{\Omega,\Gamma}} \ge \frac{\delta}{2}\|\bar\Psi^{*t}\|^2_{\mathcal{M}^1_{\Omega,\Gamma}}.
\end{equation}
Next, by virtue of the above estimate (\ref{fus-9.1}), we also have
\begin{align}
& 2\int_0^\infty \mu_\Omega(s) \left\langle {{\rm A_W^{\alpha,0,0,\omega}}}\bar\Psi^{*t}(s),\bar V^*(t) \right\rangle_{\mathbb{X}^2} {\rm d}s + 2\nu\int_0^\infty \mu_\Gamma(s)\left\langle {\rm B }\bar\psi^{*t}_\Gamma(s),\xi^t(s) \right\rangle_{L^2(\Gamma)} {\rm d}s - 2\left\langle \bar V^*(t),\bar\Psi^{*t} \right\rangle _{\mathcal{M}^1_{\Omega,\Gamma}} \notag \\
& = 2\nu\int_0^\infty \mu_\Gamma(s)\left\|\nabla_\Gamma\bar\psi^{*t}_\Gamma(s)\right\|^2_{L^2(\Gamma)}ds+2\beta\nu\int_0^\infty \mu_\Gamma(s)\left\|\bar\psi^{*t}_\Gamma(s)\right\|^2_{L^2(\Gamma)}{\rm d}s \notag \\
& - 2\nu \int_0^\infty \mu_\Gamma(s) \langle \nabla_\Gamma \bar v^*_\Gamma(t), \nabla_\Gamma\bar\psi^{*t}_\Gamma(s)\rangle_{L^2(\Gamma)} {\rm d}s -  2\beta\nu \int_0^\infty \mu_\Gamma(s)  \langle \bar v^*_\Gamma(t),\bar\psi^{*t}_\Gamma(s) \rangle_{L^2(\Gamma)} {\rm d}s \notag \\
& \ge - \frac{\nu}{2}\int_0^\infty\mu_\Gamma(s)\left\|\nabla_\Gamma \bar v^*_\Gamma\right\|^2_{L^2(\Gamma)}{\rm d}s -\frac{\beta\nu}{2}\int_0^\infty\mu_\Gamma(s)\|\bar v^*_\Gamma\|^2_{L^2(\Gamma)}{\rm d}s \notag \\
& \ge -m_\Gamma\frac{\nu}{2} \|\nabla_\Gamma \bar v^*_\Gamma\|^2_{L^2(\Gamma)} - m_\Gamma\frac{\beta\nu}{2} \|\bar v^*_\Gamma\|^2_{L^2(\Gamma)} \notag \\
& \ge -m_\Gamma\frac{\nu}{2}\|\bar V^*\|^2_{\mathbb{V}^1}. \label{dc6}
\end{align}
Combining (\ref{dc3})-(\ref{dc6}) produces
\begin{align}
\frac{{\rm d}}{{\rm d}t} & \left\{ \|\bar V^*\|^2_{\mathbb{X}^2} + \|\bar\Psi^{*t}\|^2_{\mathcal{M}^1_{\Omega,\Gamma}} \right\} + \left( 2-m_\Gamma\frac{\nu}{2} \right)\|\bar V^*\|^2_{\mathbb{V}^1}  \le 0.  \label{dc7}
\end{align}
Recall, $m_\Gamma:=\int_0^\infty\mu_\Gamma(s){\rm d}s.$
Applying the embedding $\mathbb{V}^1\hookrightarrow \mathbb{X}^2$; i.e., 
\begin{align}
\|\bar V^*\|^2_{\mathbb{X}^2} \le C_{\overline{\Omega}} \|\bar V^*\|^2_{\mathbb{V}^1},  \label{dc8}
\end{align}
and according to assumption (\ref{assk2}), we find there is a constant $m_0>0$, suitably small, so that (\ref{dc7})-(\ref{dc8}) become, for almost all $t\ge 0$,
\begin{align*}
\frac{{\rm d}}{{\rm d}t} & \left\{ \|\bar V^*\|^2_{\mathbb{X}^2} + \|\bar\Psi^{*t}\|^2_{\mathcal{M}^1_{\Omega,\Gamma}} \right\} + m_0\left( \|\bar V^*\|^2_{\mathbb{X}^2} + \|\bar\Psi^{*t}\|^2_{\mathcal{M}^1_{\Omega,\Gamma}} \right) \le 0.  \label{dc9}
\end{align*}
After applying a Gr\"{o}nwall inequality, we have that for all $t\ge 0$,
\begin{equation*}  \label{to-C2-L}
\left\| \left( \bar V^*(t),\bar\Psi^{*t} \right) \right\|_{\mathcal{H}^{0,1}_{\Omega,\Gamma}} \le \left\|\bar\Upsilon_0 \right\|_{\mathcal{H}^{0,1}_{\Omega,\Gamma}} e^{-m_0 t/2}.
\end{equation*}
Set $t^\ast:=\max\{ t_0,\frac{2}{m_0}\ln 4 \}$ (recall $t_0$ was defined in (\ref{time0}) in the proof of Lemma \ref{t:weak-ball}).
Then, for all $t\ge t^\ast$, (\ref{H2-L}) holds with $L=\bar\Lambda(t^\ast)=(\bar V^*(t^\ast),\bar\Psi^{*t^*}),$ and 
\begin{equation*}
\ell^\ast= e^{-m_0 t^\ast/2} < \frac{1}{2}.
\end{equation*}
This completes Step 1 of the proof.

{\bf Step 2.} Proof of (\ref{H2-K}).
We begin by multiplying equation (\ref{wsol1}) by $\bar W$ in $\mathbb{X}^2$, then, we multiply the equation (\ref{wsol2}), by ${\rm A_W^{\alpha,0,0,\omega}}\bar\Theta^t$ in $\mathcal{M}^0_{\Omega,\Gamma}=L^2_{\mu_\Omega\oplus\mu_\Gamma}(\mathbb{R}_+;\mathbb{X}^2)$. 
This leaves us with the two identities,
\begin{align}
\left \langle \partial_t\bar W,\bar W \right\rangle_{\mathbb{X}^2} & + \left\langle {\rm A_W^{0,\beta,\nu,\omega}}\bar W,\bar W \right\rangle_{\mathbb{X}^2} + \int_0^\infty \mu_\Omega(s) \left\langle {{\rm A_W^{\alpha,0,0,\omega}}}\bar\Theta^t(s), \bar W(t) \right\rangle_{\mathbb{X}^2} {\rm d}s  \notag \\
&  + \nu\int_0^\infty \mu_\Gamma(s) \left\langle {\rm B}\theta_\Gamma^t(s), \bar W(t) \right\rangle_{\mathbb{X}^2} {\rm d}s + \left\langle F(U_1)-F(U_2),\bar W \right\rangle_{\mathbb{X}^2}  = 0.   \label{vid-1}
\end{align}
and 
\begin{align}
\left\langle \partial_t\bar{\Theta}^t, {\rm A_W^{\alpha,0,0\omega}}\bar\Theta^t \right\rangle_{\mathcal{M}^0_{\Omega,\Gamma}} = \left\langle {\rm T_r} \bar{\Theta}^t,{\rm A_W^{\alpha,0,0\omega}}\bar\Theta^t \right\rangle_{\mathcal{M}^0_{\Omega,\Gamma}} + \left\langle \bar{W}(t), {\rm A_W^{\alpha,0,0\omega}}\bar\Theta^t \right\rangle_{\mathcal{M}^0_{\Omega,\Gamma}}.  \label{vid-2} 
\end{align}
By estimating along the lines described in various arguments already made above, we can find the following,
\begin{align}
\left\langle \partial_t\bar W,\bar W \right\rangle_{\mathbb{X}^2} + \left\langle \partial_t\bar{\Theta}^t, {\rm A_W^{\alpha,0,0\omega}}\bar\Theta^t \right\rangle_{\mathcal{M}^0_{\Omega,\Gamma}} = \frac{1}{2}\frac{{\rm d}}{{\rm d}t} \left\{ \left\|\bar W\right\|^2_{\mathbb{X}^2} + \left\|\bar\Theta^t\right\|^2_{\mathcal{M}^0_{\Omega,\Gamma}} \right\}, \label{vid-2.1}
\end{align}
\begin{align}
\left\langle {\rm A_W^{0,\beta,\nu,\omega}}\bar W,\bar W \right\rangle_{\mathbb{X}^2} = \left\|\bar W\right\|^2_{\mathbb{V}^1}, \label{vid-2.2}
\end{align}
\begin{align}
-\left\langle {\rm T_r} \bar{\Theta}^t,{\rm A_W^{\alpha,0,0\omega}}\bar\Theta^t \right\rangle_{\mathcal{M}^0_{\Omega,\Gamma}} \ge \frac{\delta}{2}\left\|\bar{\Theta}^t\right\|^2_{\mathcal{M}^0_{\Omega,\Gamma}}, \label{vid-2.3}
\end{align}
and 
\begin{align}
& \int_0^\infty \mu_\Omega(s) \left\langle {{\rm A_W^{\alpha,0,0,\omega}}}\bar\Theta^t(s), \bar W(t) \right\rangle_{\mathbb{X}^2} {\rm d}s + \nu\int_0^\infty \mu_\Gamma(s) \left\langle {\rm B}\theta_\Gamma^t(s), \bar w_\Gamma(t) \right\rangle_{L^2(\Gamma)} {\rm d}s - \left\langle \bar{W}(t), {\rm A_W^{\alpha,0,0,\omega}}\bar\Theta^t \right\rangle_{\mathcal{M}^0_{\Omega,\Gamma}} \notag \\
& \ge -m_\Gamma\frac{\nu}{2}\|\bar W\|^2_{\mathbb{V}^1}. \label{vid-2.4}
\end{align}
Using assumptions (\ref{assm-1}) and (\ref{assm-2}) with initial datum taken in the bounded set $\mathcal{B}^0$ and the uniform bound (\ref{weak-bound}), we now estimate the nonlinear terms using the estimates
\begin{align}
\langle f(u_1)-f(u_2),\bar{w}\rangle_{L^2(\Omega)} & \le \|(f(u_1)-f(u_2))\bar{w}\|_{L^1(\Omega)}  \notag \\ 
& \le \|f(u_1)-f(u_2)\|_{L^{6/5}(\Omega)}\|\bar{w}\|_{L^6(\Omega)}  \notag \\ 
& \le \ell_1\|(u_1-u_2)(1+|u_1-u_2|^2)\|_{L^{6/5}(\Omega)}\|\bar{w}\|_{L^6(\Omega)}  \notag \\ 
& \le \ell_1\|u_1-u_2\|_{L^{6}(\Omega)}\left(1+\|u_1-u_2\|^2_{L^3(\Omega)}\right)\|\bar{w}\|_{L^6(\Omega)}  \notag \\ 
& \le C\|\bar{w}\|_{H^1(\Omega)},  \label{func-1}
\end{align}
where $C=C(\ell_1,\Omega,\widetilde{P}_0)>0$ and the last inequality follows from the fact that $H^1(\Omega)\hookrightarrow L^6(\Omega)$ and $H^1(\Omega)\hookrightarrow L^3(\Omega)$.
Similarly for $\widetilde{g}$ (here the estimate is easier because $H^1(\Gamma)\hookrightarrow L^p(\Gamma)$ for any $1\le p<\infty$ as $\Gamma$ is two dimensional), 
\begin{align}
\langle \widetilde{g}(u)-\widetilde{g}(v),\bar{w}\rangle_{L^2(\Gamma)} & \le C\|\bar{w}\|_{H^1(\Gamma)}.  \label{func-2}
\end{align}
Hence, (\ref{func-1}) and (\ref{func-2}) show that, for any $\iota>0,$
\begin{align}
\left|\left\langle F(U_1)-F(U_2),\bar{W} \right\rangle_{\mathbb{X}^2}\right| & \le C_\iota\left\| \bar\Upsilon_0 \right\|^2_{\mathcal{H}^{0,1}_{\Omega,\Gamma}} + \iota\left\| \bar{W} \right\|^2_{\mathbb{V}^1}.  \label{vid-2.5}
\end{align}
Now combining (\ref{vid-1})-(\ref{vid-2.5}) and once more the embedding $\mathbb{V}^1\hookrightarrow \mathbb{X}^2$, we arrive at the inequality
\begin{align}
\frac{1}{2}\frac{{\rm d}}{{\rm d}t} \left\{ \left\|\bar W\right\|^2_{\mathbb{X}^2} + \left\|\bar\Theta^t\right\|^2_{\mathcal{M}^0_{\Omega,\Gamma}} \right\} + C^{-1}_{\overline{\Omega}}\left( 1-m_\Gamma\frac{\nu}{2}-\iota \right) \left\|\bar W\right\|^2_{\mathbb{X}^2} + \frac{\delta}{2}\left\|\bar{\Theta}^t\right\|^2_{\mathcal{M}^0_{\Omega,\Gamma}} \le C_\iota\left\| \bar\Upsilon_0 \right\|^2_{\mathcal{H}^{0,1}_{\Omega,\Gamma}}.  \label{vid-3}
\end{align}
By assumption (\ref{assk2}) there is $\iota>0$, sufficiently small, such that we may set
\[
m_1=\min\left\{ 2C^{-1}_{\overline{\Omega}}\left( 1-m_\Gamma\frac{\nu}{2}-\iota \right), \delta \right\}>0
\]
and in which (\ref{vid-3}) becomes
\begin{align}
\frac{{\rm d}}{{\rm d}t} \left\{ \left\|\bar W\right\|^2_{\mathbb{X}^2} + \left\|\bar\Theta^t\right\|^2_{\mathcal{M}^0_{\Omega,\Gamma}} \right\} + m_1\left( \left\|\bar W\right\|^2_{\mathbb{X}^2} + \left\|\bar{\Theta}^t\right\|^2_{\mathcal{M}^0_{\Omega,\Gamma}} \right) \le C\left\| \bar\Upsilon_0 \right\|^2_{\mathcal{H}^{0,1}_{\Omega,\Gamma}}.  \label{vid-3.1}
\end{align}
Integrating (\ref{vid-3.1}) with respect to $t$ in $[0,T]$, for some fixed $0<T<\infty$, we obtain
\begin{align*}
\left\|\left(\bar W(t), \bar\Theta^t \right)\right\|_{\mathcal{H}^{0,1}_{\Omega,\Gamma}} \le C(T)\left\| \bar\Upsilon_0 \right\|_{\mathcal{H}^{0,1}_{\Omega,\Gamma}}.  \notag
\end{align*}

By following the proof of Lemma \ref{t:quasi-ball}, it also follows that $\|\bar W(t^*)\|_{\mathbb{V}^1}\le K$ for some constant $K>0$, independent of $t.$
Hence, the following bound further follows from Lemmata \ref{what-1} and \ref{what-2} (cf. (\ref{fus-13})),
\begin{align}
\left\|{\rm T_r}\bar\Theta^t\right\|^2_{\mathcal{M}^1_{\Omega,\Gamma}} + \sup_{\tau\ge1} \tau\mathbb{T}(\tau;\bar\Theta^t) & \le C(T)\left\| \bar\Upsilon_0 \right\|^2_{\mathcal{H}^{0,1}_{\Omega,\Gamma}} \left( e^{-\min\{\delta_\Omega,\delta_\Gamma\} t}(t+1) + 1 \right)  \notag \\
& \le C(T)\left\| \bar\Upsilon_0 \right\|^2_{\mathcal{H}^{0,1}_{\Omega,\Gamma}}.   \notag
\end{align}
Thus, letting $T=t^*$ (from Step 1), we obtain, 
\begin{align}
\left\|\left(\bar W(t^*), \bar\Theta^{t^*} \right)\right\|_{\mathcal{H}^{0,1}_{\Omega,\Gamma}} \le C\left\| \bar\Upsilon_0 \right\|_{\mathcal{H}^{0,1}_{\Omega,\Gamma}}.  \notag
\end{align}
Inequality (\ref{H2-K}) now follows with $K= \bar\Xi(t^*)=(\bar{W}(t^*),\bar{\Theta}^{t^*})$. 
This finishes the proof of (H2).
\end{proof}

\begin{lemma} \label{t:to-H3} 
Suppose the assumptions of Theorem \ref{t:quasi-strong} and Lemma \ref{t:weak-ball} hold.
Then condition (H3) holds.
\end{lemma}

\begin{proof}
Let $R>0$ and $\Upsilon_{01}=(U_{01},\Phi_{01}),\Upsilon_{02}=(U_{02},\Phi_{02}) \in\mathbb{V}^2\times (\mathcal{M}^2_{\Omega,\Gamma} \cap {\rm D(T_r)})$ be such that $\|\Upsilon_{01}\|_{\widehat{\mathcal{H}}^{0,1}_{\Omega,\Gamma}} \le R$ and $\|\Upsilon_{02}\|_{\widehat{\mathcal{H}}^{0,1}_{\Omega,\Gamma}} \le R$.
Let $t_1,t_2\in[t^*,2t^*]$.
In the norm of $\mathcal{H}^{-1,0}_{\Omega,\Gamma}$, we calculate
\begin{align}
& \left\|\mathcal{S}(t_1)\Upsilon_{01}-\mathcal{S}(t_2)\Upsilon_{02}\right\|_{\mathcal{H}^{-1,0}_{\Omega,\Gamma}}  \notag \\ 
& \le \left\|\mathcal{S}(t_1)\Upsilon_{01}-\mathcal{S}(t_1)\Upsilon_{02}\right\|_{\mathcal{H}^{-1,0}_{\Omega,\Gamma}} + \left\|\mathcal{S}(t_1)\Upsilon_{02}-\mathcal{S}(t_2)\Upsilon_{02}\right\|_{\mathcal{H}^{-1,0}_{\Omega,\Gamma}}.  \label{h3-1}
\end{align}
The first term on the right-hand side of (\ref{h3-1}) is bounded uniformly in $t$ on compact intervals by (\ref{lowc-1}).
Also, directly from (\ref{weak-bound}) there holds, 
\begin{equation*}
\left\| \mathcal{S}(t)\Upsilon_0 \right\|_{\mathcal{H}^{-1,0}_{\Omega,\Gamma}} \le P_0,
\end{equation*}
but where now the size of the initial datum, $R$, depends on the size of $\widehat{\mathcal{B}}^0$. 
Hence, on the compact interval $[t^{\ast },2t^{\ast }]$, the map $t\mapsto S(t)\Upsilon_0$ is Lipschitz continuous for each fixed $\Upsilon_0\in \widehat{\mathcal{B}}^0$.
This means there is a constant $L=L(t^{\ast })>0$ such that
\begin{equation*}
\left\|\mathcal{S}(t_1)\Upsilon_0 - \mathcal{S}(t_2)\Upsilon_0 \right\|_{\mathcal{H}^{-1,0}_{\Omega,\Gamma}} \leq L|t_1-t_2|.
\end{equation*}
Therefore, (C3) follows.
This concludes the proof. 
\end{proof}

\begin{remark}
According to Proposition \ref{abstract1} the semigroup of solution operators $\mathcal{S}(t):\widehat{\mathcal{H}}^{0,1}_{\Omega,\Gamma}\rightarrow \widehat{\mathcal{H}}^{0,1}_{\Omega,\Gamma}$ possesses a finite dimensional exponential attractor $\mathcal{E}^{-1}$ that is bounded in $\widehat{\mathcal{H}}^{0,1}_{\Omega,\Gamma}$, $\mathcal{E}^{-1}\subset \widehat{\mathcal{B}}^0$, compact in $\mathcal{H}^{-1,0}_{\Omega,\Gamma}$, and which attracts bounded subsets of $\widehat{\mathcal{B}}^0$ exponentially fast (in the topology of $\widehat{\mathcal{H}}^{0,1}_{\Omega,\Gamma}$). 
In order to show that the attraction property (3) in Theorem \ref{t:wk-exp-attr} also holds---that is, in order to show that the basin of attraction of $\mathcal{E}^{-1}$ is all of $\widehat{\mathcal{H}}^{0,1}_{\Omega,\Gamma}$---we appeal to the transitivity of the exponential attraction in Proposition \ref{t:exp-attr}.
We already know that the exponential attractor $\mathcal{E}^{-1}$ attracts the bounded absorbing set $\widehat{\mathcal{B}}^0$ exponentially.
So it suffices to show that the absorbing set $\widehat{\mathcal{B}}^0$ also attracts all bounded subsets of $\widehat{\mathcal{H}}^{0,1}_{\Omega,\Gamma}$ exponentially.
We only need to recall Lemma \ref{t:quasi-ball} where the rate of attraction is given in (\ref{strong-decay}).
Hence, we can find $Q(\cdot)>0$ and $\nu_0>0$ so that, for all $t\ge0$, there holds
\begin{equation*}
\mathrm{dist}_{\widehat{\mathcal{H}}^{0,1}_{\Omega,\Gamma}}(\mathcal{S}(t)B,\widehat{\mathcal{B}}^0) \le Q(\|B\|_{\widehat{\mathcal{H}}^{0,1}_{\Omega,\Gamma}}) e^{-\nu_0 t}.
\end{equation*}
The desired attraction property follows from the embedding $\widehat{\mathcal{H}}^{0,1}_{\Omega,\Gamma}\hookrightarrow\mathcal{H}^{-1,0}_{\Omega,\Gamma}$.
\end{remark}

\appendix
\section{}

For the reader's convenience we report some important results that are needed in the article.

The following lemma is from \cite[Lemma 2.2]{Gal&Grasselli08}. It is in the spirit of the $H^s$-elliptic regularity estimate that can be found in \cite[Theorem II.5.1]{Lions&Magenes72}. 

\begin{lemma}  \label{t:appendix-lemma-3}
Consider the linear boundary value problem,
\begin{equation*}  \label{appendix-BVP}
\left\{\begin{array}{rl}
-\Delta u+\alpha u & = \psi_1 \quad\text{in}\quad\Omega, \\ 
-\Delta_{\Gamma}u + \partial_{\mathbf{n}} u + \beta u & = \psi_2 \quad\text{on}\quad\Gamma.
\end{array}\right.
\end{equation*}
If $(\psi_1,\psi_2)^{\rm tr}\in H^s(\Omega)\times H^s(\Gamma)$, for $s\geq 0$ and $s+\frac{1}{2} \not\in\mathbb{N}$, then the following estimate holds for some constant $C>0$,
\begin{equation*}  \label{H2-regularity-estimate}
\|u\|_{H^{s+2}(\Omega)} + \|u\|_{H^{s+2}(\Gamma)} \leq C\left( \|\psi_1\|_{H^s(\Omega)} + \|\psi_2\|_{H^s(\Gamma)} \right).
\end{equation*}
\end{lemma}

The following result is the so-called transitivity property of exponential attraction from \cite[Theorem 5.1]{FGMZ04}.

\begin{proposition}  \label{t:exp-attr}
Let $(\mathcal{X},d)$ be a metric space and let $S_t$ be a semigroup acting on this space such that 
\[
d(S_t x_1,S_t x_2) \le C e^{Kt} d(x_1,x_2),
\]
for appropriate constants $C$ and $K$. 
Assume that there exists three subsets $U_1$,$U_2$,$U_3\subset\mathcal{X}$ such that 
\[
{\rm dist}_\mathcal{X}(S_t U_1,U_2) \le C_1 e^{-\alpha_1 t}, \quad{\rm dist}_\mathcal{X}(S_t U_2,U_3) \leq C_2 e^{-\alpha_2 t}.
\]
Then 
\[
{\rm dist}_\mathcal{X}(S_t U_1,U_3) \le C' e^{-\alpha' t},
\]
where $C'=CC_1+C_2$ and $\alpha'=\frac{\alpha_1\alpha_1}{K+\alpha_1+\alpha_2}$.
\end{proposition}

The following statement refers to a frequently used Gr\"{o}nwall-type inequality that is useful when working with dissipation arguments. 
We also refer the reader to \cite[Lemma 2.1]{Conti-Pata-2005}, \cite[Lemma 2.2]{Grasselli&Pata02}, \cite[Lemma 5]{Pata&Zelik06-2}.

\section*{Acknowledgments}

The author gratefully acknowledges the anonymous referees for their careful reading of the manuscript and their many insightful suggestions. 

\providecommand{\bysame}{\leavevmode\hbox to3em{\hrulefill}\thinspace}
\providecommand{\MR}{\relax\ifhmode\unskip\space\fi MR }
\providecommand{\MRhref}[2]{%
  \href{http://www.ams.org/mathscinet-getitem?mr=#1}{#2}
}
\providecommand{\href}[2]{#2}

\end{document}